\documentclass{amsart}
\usepackage[utf8]{inputenc}\usepackage{amssymb}
\usepackage{amsmath}
\usepackage{graphicx}
\usepackage{manfnt}
\usepackage{amsthm}
\usepackage[mathscr]{euscript}
\usepackage{mathtools}
\usepackage{epstopdf}
\usepackage{bigints}
\usepackage{color}
\usepackage{dsfont}
\usepackage{mathabx}
\usepackage{todonotes}
\usepackage{comment}

\usepackage{tikz}
\usepackage{tikz-cd}
\usepackage{bm}
\usepackage{bbm}
\usetikzlibrary{arrows}
\usetikzlibrary{positioning}
\usepackage{mathrsfs}
\numberwithin{equation}{section}
\usepackage{hyperref}
\usepackage{enumitem}

\newcommand{\cO}{\mc{O}}

\newcommand{\N}{\mathbb{N}}
\newcommand{\C}{\mathbb{C}}
\newcommand{\Z}{\mathbb{Z}}

\newcommand{\A}{\mathbb{A}}
\newcommand{\mf}{\mathfrak}
\newcommand{\mc}{\mathcal}
\newcommand{\bp}{\begin{pmatrix}}
\newcommand{\ep}{\end{pmatrix}}

\newcommand{\g}{\mf{g}}

\newcommand{\n}{\mf{n}}
\newcommand{\h}{\mf{h}}
\newcommand{\levi}{\mf{l}}
\newcommand{\bimod}[2]{#1\mbox{-mod-}#2}
\newcommand{\gmodT}{\bimod{\mf{g}}{T}}

\DeclareMathOperator{\id}{id}

\DeclareMathOperator{\ch}{ch}
\DeclareMathOperator{\End}{End}
\DeclareMathOperator{\Hom}{Hom}
\DeclareMathOperator{\Ann}{Ann}

\DeclareMathOperator{\stab}{stab}

\DeclareMathOperator{\MT}{MT}
\DeclareMathOperator{\MTM}{MTM}
\newcommand{\realmtm}{\Upsilon}
\DeclareMathOperator{\Lie}{Lie}

\DeclareMathOperator{\im}{im}

\DeclareMathOperator{\Int}{Int}
\DeclareMathOperator{\gr}{gr}
\DeclareMathOperator{\Wh}{Wh}

\DeclareMathOperator{\coker}{coker}
\DeclareMathOperator{\can}{can}

\theoremstyle{plain}
\newtheorem{theorem}{Theorem}[subsection]
\theoremstyle{definition}
\newtheorem{definition}[theorem]{Definition}
\theoremstyle{plain}
\newtheorem{lemma}[theorem]{Lemma}
\theoremstyle{plain}
\newtheorem{proposition}[theorem]{Proposition}
\theoremstyle{remark}

\theoremstyle{plain}

\theoremstyle{remark}
\newtheorem{remark}[theorem]{Remark}
\theoremstyle{plain}

\theoremstyle{definition}

\title{Two proofs of a Jantzen Conjecture for Whittaker Modules}
\author{Jens Niklas Eberhardt, Anna Romanov}

\begin{document}

\maketitle

\begin{abstract}
We define a filtration of a standard Whittaker module over a complex semisimple Lie algebra and and establish its fundamental properties. Our filtration specialises to the Jantzen filtration of a Verma module for a certain choice of parameter. We prove that embeddings of standard Whittaker modules are strict with respect to our filtration, and that the filtration layers are semisimple. This provides a generalisation of the Jantzen conjectures to Whittaker modules. We prove these statements in two ways. First, we give an algebraic proof which compares Whittaker modules to Verma modules using a functor introduced by Backelin. Second, we give a  geometric proof using mixed twistor $\mc{D}$-modules.
\end{abstract}

\tableofcontents

\section{Introduction}

\subsection{Overview} 
In this paper, we state and prove a Jantzen conjecture for Whittaker modules, generalising a famous theorem of Beilinson--Bernstein. We provide two proofs --- one algebraic (relating Whittaker modules to Verma modules in category $\mc{O}$), and one geometric (using mixed twistor $\mc{D}$-modules). 

The Jantzen conjectures for Verma modules are a classical result in the representation theory of semisimple Lie algebras. They establish that a certain algorithmically-defined filtration of a Verma module satisfies remarkable functoriality and semisimplicity properties. The conjectures were proven by Beilinson--Bernstein in \cite{BBJantzen} using mixed structures on categories of holonomic $\mc{D}$-modules. 

We define an analogous filtration on a standard Whittaker module. Our main result (Theorem \ref{thm: Jantzen conjecture Whittaker}) is that our filtration satisfies the same functoriality and semisimplicity properties as the Jantzen filtration of a Verma module. Our theorem contains the Jantzen conjectures for Verma modules as a special case. 

 This paper completes a program set out in \cite{Romanov, BrownRomanov, BRpairings} to establish which results from category $\mc{O}$ can be lifted to the Whittaker category $\mc{N}$. Previous work on Whittaker modules is based on two main approaches. The first approach is algebraic, relating Whittaker modules to category $\mc{O}$, then using known results about highest weight modules to deduce properties of Whittaker modules. The second approach is geometric, using Beilinson--Bernstein localisation to relate Whittaker modules to twisted equivariant $\mc{D}$-modules on the associated flag variety. In this paper we use both approaches, developing significant extensions of existing techniques in each setting. One notable aspect of our geometric proof is that it relies on the machinery of mixed twistor $\mc{D}$-modules developed in  \cite{sabbahPolarizableTwistorMathcal2005, sabbahWildTwistorMathcal2009,mochizukiWildHarmonicBundles2011,mochizukiMixedTwistorDmodules2015}. Mixed twistor $\mc D$-modules can be seen as an instance of Simpson's meta theorem \cite{simpsonMixedTwistorStructures1997} applied to Saito's mixed Hodge modules \cite{Saito1,Saito2}. Our work provides one of the first applications of this technology to geometric representation theory.

In the remainder of the introduction, we explain our results in more detail.

\subsection{Whittaker modules}
\label{sec: Whittaker modules intro}
Our theorem deals with a class of representations called standard Whittaker modules. In this section, we describe their construction and basic properties. 

Let $\mf{g}$ be a complex semisimple Lie algebra, and $\mf{h} \subset \mf{b} \subset \mf{g}$ fixed Cartan and Borel subalgebras. Let $\mf{n}=[\mf{b},\mf{b}]$ be the nilpotent radical of $\mf{b}$. An {\em $\eta$-Whittaker module}\footnote{We drop the $\eta$ and use the term {\em Whittaker module} when the character $\eta$ is clear from context.} is a $\mf{g}$-module which is cyclically generated by a vector on which $\mf{n}$ acts by a character $\eta:\mf{n} \rightarrow \C$.

When $\eta$ is the zero character, Verma modules, highest weight modules, and simple finite-dimensional $\mf{g}$-modules satisfy this condition, providing examples of $0$-Whittaker modules. At the other extreme, when $\eta$ is nondegenerate (meaning that it does not vanish on any simple root space in $\mf{n}$) Kostant showed that $\eta$-Whittaker modules have an elementary description: simple nondegenerate Whittaker modules are uniquely determined by their infinitesimal character, and equivalence classes of nondegenerate Whittaker modules are in bijection with ideals in the center of the universal enveloping algebra of $\g$ \cite{Kostant}.

For partially degenerate $\eta$, \cite{McDowell} provided a classification of simple $\eta$-Whittaker modules as unique irreducible quotients of {\em standard Whittaker modules}, which are defined as follows. Let $\Pi_\eta$ be the subset of simple roots for which $\eta$ does not vanish on the corresponding root space. The set $\Pi_\eta$ determines a Levi subalgebra and parabolic subalgebra\footnote{See \S \ref{sec: definitions}-\ref{sec: standard/costandard modules and contravariant pairings} for precise definitions and $\rho$-shift conventions.} $\mf{l}_\eta\subset \mf{p}_\eta \subset \g$. Denote by $Z(\mf{l}_\eta)\subset U(\mf{l}_\eta)$ the center of the universal enveloping algebra of $\mf{l}_\eta$ and $\mf{n}_\eta =  \mf{n} \cap \mf{l}_\eta$. Let $\rho, \rho_\eta \in \h^*$ be the half-sum of positive roots of $\g$ and $\mf{l}_\eta$, respectively. For $\lambda \in \h^*$, the corresponding standard Whittaker module is 
\begin{equation}
    \label{eq: standard Whittaker}
M(\lambda, \eta):= U(\g) \otimes_{U(\mf{p}_\eta)} U(\mf{l}_\eta) \otimes_{Z(\mf{l}_\eta) \otimes U(\mf{n}_\eta)} \C_{\lambda - \rho + \rho_\eta, \eta}.
\end{equation}
Here $\C_{\lambda - \rho + \rho_\eta}$ is the one-dimensional $Z(\mf{l}_\eta) \otimes U(\mf{n}_\eta)$-module on which $Z(\mf{l}_\eta)$ acts by the infinitesimal character corresponding to $\lambda - \rho + \rho_\eta$, and $U(\mf{n}_\eta)$ acts by $\eta$. 

When $\eta=0$, the standard Whittaker module $M(\lambda, 0)$ in \eqref{eq: standard Whittaker} is the Verma module of highest weight $\lambda - \rho$, and the corresponding simple Whittaker module is its unique irreducible quotient. When $\eta$ is non-degenerate, the standard Whittaker module in \eqref{eq: standard Whittaker} is irreducible for all choices of $\lambda$. For partially degenerate $\eta$, the standard Whittaker module interpolates between these two extremes. 

McDowell's classification scheme bears a strong resemblance to the classification of simple highest weight modules in category $\mc{O}$, and even includes it as a special case. Motivated by this analogy, Mili{\v{c}}i{\'c}--Soergel initiated a categorical approach to Whittaker modules in \cite{CompositionSeries}, defining a category $\mc{N}$ whose simple objects are precisely the simple Whittaker modules. Working within this category, Mili{\v{c}}i{\'c}--Soergel and Backelin \cite{Backelin} provided a description of composition series multiplicities of standard Whittaker modules in terms of parabolic Kazhdan--Lusztig polynomials, further demonstrating the similarities between the categories $\mc{O}$ and $\mc{N}$.  

In \cite{TwistedSheaves}, Mili{\v{c}}i{\'c}--Soergel proposed a geometric approach to the study of category $\mc{N}$, using Beilinson--Bernstein localisation to relate Whittaker modules to twisted equivariant $\mc{D}$-modules on the flag variety. This approach allowed results about $\mc{N}$ to be proven directly, instead of by reduction to known results in $\mc{O}$, which was the proof strategy in \cite{CompositionSeries, Backelin}. The geometric approach was further developed in \cite{Romanov} and \cite{Zhao} where a geometric algorithm for computing composition multiplicities was established that provides a direct $\mc{D}$-module proof of the results in \cite{CompositionSeries, Backelin}. 

The parallel algebraic and geometric approaches to studying category $\mc{N}$ have demonstrated many similarities with category $\mc{O}$. However, there are two major departures from highest weight theory which provide obstacles to furthering the study of Whittaker modules:
\begin{enumerate}
    \item Unlike Verma modules, standard Whittaker modules admit many distinct contravariant forms \cite{BrownRomanov}. This makes it difficult to generalise arguments about Verma modules which rely on the uniqueness of the Shapovalov form \cite{Shapovalov}. 
    \item The $\mc{D}$-modules corresponding to Whittaker modules have irregular singularities \cite{TwistedSheaves}. This prevents the use of the Riemann--Hilbert correspondence to compare Whittaker modules to categories of perverse sheaves, and, in particular, it removes the possibility of studying them using the mixed $\ell$-adic sheaves of \cite{Deligne} or the mixed Hodge modules of \cite{Saito1, Saito2}. 
\end{enumerate}

The obstacles above made one fundamental result in category $\mc{O}$ seem especially out of reach - the Jantzen conjectures. The Jantzen conjectures describe the properties of the Jantzen filtration of a Verma module, which is defined using the unique Shapovalov form. By obstacle (1) above, there is no clear generalisation of this filtration to standard Whittaker modules since there is no unique contravariant form. Moreover, the proof of the Jantzen conjectures by Beilinson--Bernstein in \cite{BBJantzen} relies on passage to categories of mixed $\ell$-adic sheaves or mixed Hodge modules, so obstacle (2) eliminates the possibility of a direct generalisation. 

In this paper, we explain how one can circumvent obstacles (1) and (2) to both state and prove a Jantzen conjecture for Whittaker modules. We use two key tools. The first is the unique contravariant pairing between a standard Whittaker module and a Verma module introduced in \cite{BRpairings}, which gives us an appropriate substitute for the Shapovalov form. This allows us to overcome obstacle (1). The second is the theory of mixed twistor $\mc{D}$-modules \cite{mochizukiMixedTwistorDmodules2015}, which encapsulates holonomic $\mc{D}$-modules with irregular singularities and serves as a replacement for mixed Hodge modules. This allows us to overcome obstacle (2). 

We describe our generalisation of the Jantzen conjectures in more detail in the following section.

\subsection{A Jantzen conjecture}
\label{sec: A Jantzen conjecture intro}

A key result in \cite{BRpairings} is that blocks of category $\mc{N}$ have the structure of highest weight categories. In particular, for each $\lambda \in \h^*$ and $\mf{n}$-character $\eta$, there exists a standard Whittaker module $M(\lambda, \eta)$, a costandard module $M^\vee(\lambda, \eta)$, and a canonical map 
\[
\psi: M(\lambda, \eta) \rightarrow M^\vee(\lambda, \eta)
\]
in category $\mc{N}$. The canonical map has an explicit realisation in terms of the unique contravariant pairing between $M(\lambda, \eta)$ and the Verma module $M(\lambda)$ (Proposition \ref{prop: properties of costandards}). 

To define the Jantzen filtration, we need a deformation of the above set-up in a specified direction $\gamma \in \h^*$. Let $T=\mc{O}(\C \gamma)$ be the ring of regular functions on the line $\C \gamma \subset \h^*$. Identify $T$ with the polynomial ring $\C[s]$, and set $A = T_{(s)}$ to be the localisation of $T$ with respect to the prime ideal $(s)$. In \S\ref{sec: deformed modules}, we define deformed standard and costandard Whittaker modules, which are $(\g,A)$-bimodules admitting a canonical $(\g, A)$-bimodule homomorphism 
\[
\psi_A: M_A(\lambda, \eta) \rightarrow M_A^\vee(\lambda, \eta).
\]
Setting $s=0$ recovers the non-deformed set-up above, and setting $\eta = 0$ recovers the canonical map from deformed Verma modules to deformed dual Verma modules studied in \cite{Soergel}. 

The deformed module $M^\vee_A(\lambda, \eta)$ has a $(\g,A)$-module filtration given by powers of $s$, and the {\em Jantzen filtration} of $M_A(\lambda, \eta)$ is defined to be the filtration obtained by pulling back this ``powers of $s$'' filtration along the canonical map $\psi_A$:
\[
M_A(\lambda, \eta)_{-i}=\{ m \in M_A(\lambda, \eta) \mid \psi_A(m) \in M_A^\vee(\lambda, \eta)s^i\}.
\]
Setting $s=0$, we recover a filtration of $M(\lambda, \eta)$ which we call the Jantzen filtration. When $\eta=0$, this specialises to the classical Jantzen filtration of a Verma module defined in \cite{Jantzen} using the Shapovalov form.

\begin{remark} (Computability of Jantzen filtration) Because the Jantzen filtration is defined in terms of an explicitly defined contravariant pairing (see (5.1) in \cite{BRpairings}), the Jantzen filtration is directly computable. In contrast, other representation-theoretic filtrations such as the composition series are known to exist, but are difficult to compute algorithmically.
\end{remark}

Despite being computable, it is unclear from the definition what properties the Jantzen filtration should satisfy.  Jantzen conjectured that for Verma modules, his filtration satisfies strong functoriality and semisimplicity properties. In particular, the {\em Jantzen conjectures} are the statements (posed as questions in \cite[\S5.17]{Jantzen}) that for $\gamma = \rho$, embeddings of Verma modules are strict\footnote{By {\em strict} we mean that if $\iota: M(\mu) \hookrightarrow M(\lambda)$ is an embedding, the  filtration on $\iota(M(\mu))$ induced from the Jantzen filtration on $M(\lambda)$ agrees with the Jantzen filtration on $M(\mu)$ up to a shift.} for Jantzen filtrations, and the Jantzen filtration coincides with the socle filtration. In particular, the filtration layers of the Jantzen filtration of a Verma module are semisimple. 

In the years following Jantzen's conjectures, it became clear that the Jantzen filtration carries fundamental information about Verma modules and category $\mc{O}$. Work of Barbasch \cite{Barbasch}, Gabber-Joseph \cite{GJ}, and others showed that the Jantzen conjectures imply the Kazhdan--Lusztig conjectures for Verma modules. In fact, they imply a stronger statement, relating multiplicities of simple modules in layers of the Jantzen filtration to coefficients of Kazhdan--Lusztig polynomials \cite{GJ}. 

The Jantzen conjectures for Verma modules were proven by Beilinson--Bernstein in \cite{BBJantzen} using weight theory and mixed $\ell$-adic sheaves. Twenty years later, Williamson provided an algebraic proof using Soergel bimodules \cite{Williamson}, building on earlier work of K\"ubel and Soergel \cite{Kubel1, Kubel2, Soergel}. Williamson's proof also demonstrated that the Jantzen filtration depends on the deformation parameter $\gamma \in \h^*$, answering a longstanding question of Deodhar.  

The main theorem of this paper (which appears as Theorem \ref{thm: Jantzen conjecture Whittaker} in the body of the text) is a generalisation of Jantzen's conjectures to standard Whittaker modules, using the Jantzen filtration of $M(\lambda, \eta)$ defined above. 

\begin{theorem}
\label{thm: Jantzen conjecture introduction} 
   {\em  (The Jantzen conjectures for Whittaker modules) }
    \begin{enumerate}[label=(\roman*)]
        \item (Hereditary property) Embeddings of standard Whittaker modules are strict for Jantzen filtrations. 
        \item (Semisimplicity property) The filtration layers of the Jantzen filtration of $M(\lambda, \eta)$ are semisimple. 
    \end{enumerate}
\end{theorem}

When $\eta=0$, Theorem \ref{thm: Jantzen conjecture introduction} recovers the Jantzen conjectures for Verma modules. We provide two proofs of Theorem \ref{thm: Jantzen conjecture introduction}. In the following sections, we outline the strategies of each proof.

\subsection{An algebraic approach}
\label{sec: An algebraic approach intro}

Our first approach to Theorem \ref{thm: Jantzen conjecture introduction} uses a functor defined by Backelin to relate Whittaker modules to category $\mc{O}$. The functor is constructed as follows. For any $\mf{g}$-module $M$, denote by 
\[
\Gamma_\eta(M) = \{m \in M \mid (X - \eta(X))^k m = 0 \text{ for all }X \in \n \text{ and } k\gg0\}
\]
the set of generalised $\eta$-Whittaker vectors in $M$. This set is $\g$-stable \cite[Lem. 6.1]{BRpairings}. If $M = \bigoplus_{\mu \in \h^*} M_\mu$ is a weight module, then define the completion of $M$ to be the direct product of weight spaces:
\[
\overline{M} = \prod_{\mu \in \h^*} M_\mu.
\]
In \cite{Backelin}, Backelin defines a functor 
\[
\overline{\Gamma}_\eta: \mc{O} \rightarrow \mc{N}; M \mapsto \Gamma_\eta(\overline{M}). 
\]
Backelin's functor is exact and sends Verma modules to standard Whittaker modules \cite[Prop. 6.9]{Backelin}, so it is a useful tool in using relating properties of standard Whittaker modules to properties of Verma modules. 

The approach of our proof is to show that $\overline{\Gamma}_\eta$ sends the Jantzen filtration of a Verma module to the Jantzen filtration of a standard Whittaker module. This allows us to use the Jantzen conjecture for Verma modules to deduce the Jantzen conjecture for Whittaker modules. In particular, the semisimplicity in Theorem \ref{thm: Jantzen conjecture introduction}.(ii) follows immediately from the semisimplicity of filtration layers in category $\mc{O}$, and the hereditary property in Theorem \ref{thm: Jantzen conjecture introduction}.(i) follows from the fact that $\overline{\Gamma}_\eta$ is a quotient functor  when restricted to blocks of $\mc{O}$ \cite[Cor. 2]{AriasBackelin}. 

To show that $\overline{\Gamma}_\eta$ matches Jantzen filtrations, we need to define a deformed version of $\overline{\Gamma}_\eta$ and establish its properties. We do this in \S\ref{sec: deformed Backelin}. The key result is Theorem \ref{thm: deformed Vermas go to standard Whittakers}, which demonstrates that the deformed Backelin functor sends certain deformed Verma modules (those for which the corresponding $\mf{l}_\eta$-Verma module is simple) to deformed Whittaker modules. Our proof relies on a factorisation of the functor $\overline{\Gamma}_\eta$ into a non-degenerate piece and a zero piece. Combining this with a careful analysis of partial weight spaces in deformed Verma modules, we are able to deduce the theorem.

\subsection{A geometric approach}
\label{sec: a geometric approach intro}
Our second proof of Theorem \ref{thm: Jantzen conjecture introduction} is a direct proof using $\mc{D}$-modules. This proof generalises Beilinson--Bernstein's original approach in \cite{BBJantzen} by replacing mixed $\ell$-adic sheaves with mixed twistor $\mc{D}$-modules. We explain the strategy below.

An essential ingredient in the construction of Whittaker $\mc{D}$-modules is the exponential $\mc{D}$-module (see Remark \ref{rem: the exponential D-module}), which is holonomic but not regular. This prevents us from using the Riemann--Hilbert correspondence to compare Whittaker $\mc{D}$-modules to constructible sheaves and their mixed upgrades. Instead, we use the mixed twistor $\mc{D}$-modules introduced by Sabbah and Mochizuki, which encapsulate $\mc{D}$-modules with irregular singularities.

Like their mixed Hodge module counterparts, mixed twistor $\mc{D}$-modules carry weight filtrations with strong functoriality and semisimplicity properties, similar to those in the Jantzen conjectures. Moreover, for mixed twistor $\mc{D}$-modules constructed using maximal extension functors, the weight filtration agrees with the monodromy filtration, which is explicitly computable. The strategy of our proof is to show that the $\mc{D}$-modules corresponding to standard Whittaker modules fit into this framework, and thus the Jantzen filtration inherits the properties of weight filtrations.

Our first step is to upgrade the geometric results of \cite{TwistedSheaves, Romanov}, which use twisted sheaves of differential operators on the flag variety, to $H$-monodromic $\mc{D}$-modules on base affine space. We do this in \S \ref{sec: Whittaker D-modules}. Once in the setting of base affine space, we are able to realise the $\mc{D}$-modules corresponding to standard Whittaker modules as subsheaves of maximal extension functors. This endows them with a monodromy filtration.

Next, we argue that the global sections of the monodromy filtration align with the Jantzen filtration. A key component is recognising the geometric incarnation of the deformed standard Whittaker module which arises in the construction of the maximal extension functor. This is Theorem \ref{thm: deformed Vermas go to standard Whittakers}. Because this critical step is given minimal justification in \cite{BBJantzen}, we have included a careful argument that reduces to Verma modules when $\eta=0$. 

Finally, we show that Whittaker $\mc{D}$-modules admit a mixed twistor structure, and hence carry a weight filtration. This allows us to use results in \cite{mochizukiMixedTwistorDmodules2015} to establish that the weight filtration agrees with the monodromy filtration on Whittaker $\mc{D}$-modules. This completes the proof: the functoriality of the weight filtration implies the hereditary property of the Jantzen filtration (Theorem \ref{thm: Jantzen conjecture introduction}.(i)) and the semisimplicity of filtration layers of the weight filtration implies Theorem \ref{thm: Jantzen conjecture introduction}.(ii).

\subsection{Structure of the paper}
\label{sec: Structure of the paper intro}

This paper is structured as follows. 

\S\ref{sec: definitions}: We define Whittaker modules and Mili{\v{c}}i{\'c}--Soergel's category $\mc{N}$. 

\S\ref{sec: standard/costandard modules and contravariant pairings}: We define standard and costandard Whittaker modules and record their basic properties.

\S\ref{sec: deformed modules}: We define deformed category $\mc{O}$, deformed Verma/dual Verma modules, and deformed standard/costandard Whittaker modules. 

\S\ref{sec: two algebraic Jantzen filtrations}: We define the Jantzen filtration of Verma modules and standard Whittaker modules. First, it is necessary to extend the results in \cite{BRpairings} about contravariant pairings to the deformed setting. We state the Jantzen conjecture for Whittaker modules. 

\S\ref{sec: deformed Backelin}: We define a deformed version of Backelin's functor and establish its basic properties. This involves a careful analysis of the set of Whittaker vectors in completed deformed Whittaker modules. The main result is Theorem \ref{thm: deformed Vermas go to standard Whittakers}, which establishes that the deformed Backelin functor sends deformed Verma modules to deformed standard Whittaker modules. 

\S\ref{sec: Jantzen conjecture via category O}: We prove that the Backelin functor sends the Jantzen filtration of a Verma module to the Jantzen filtration of the corresponding standard Whittaker module. This allows us to prove the Jantzen conjecture algebraically.

\S\ref{sec: preliminaries}: We set our $\mc{D}$-module conventions and record  basic facts about functors. We define a map from the extended universal enveloping algebra $\widetilde{U}$ to global differential operators on base affine space $\widetilde{X}$ which gives the global sections of $\mc{D}_{\widetilde{X}}$-modules the structure of $\widetilde{U}$-modules. We introduce monodromic $\mc{D}$-modules, and explain their role in our story. 

\S\ref{sec: Whittaker D-modules}: We briefly recall the geometric approach to Whittaker modules set out in \cite{TwistedSheaves, Romanov}, then explain how this can be formulated in terms of monodromic $\mc{D}$-modules. 

\S\ref{sec: a geometric Jantzen filtration via monodromy}: We define the monodromy filtration and explain how it can be used to provide a filtration on $\mc{D}$-modules obtained via maximal extension functors. We explain how Whittaker $\mc{D}$-modules can be made to fit into this framework. 

\S\ref{sec: relationship between algebraic and geometric Janten filtrations}: We show that the global sections functor sends the monodromy filtration on a Whittaker $\mc{D}$-module to the Jantzen filtration on the standard Whittaker module. 

\S\ref{sec: background on mixed twistor modules}: We establish background and conventions for mixed twistor $\mc{D}$-modules. 

\S\ref{sec: proof of the Jantzen conjecture}: We provide our geometric proof of the Jantzen conjecture. 

\subsection{Acknowledgements}
This paper completes a project which began with the papers \cite{BrownRomanov} and \cite{BRpairings}. The second author would like to thank Adam Brown for the many hours spent together contemplating forms, duality, and Jantzen filtrations. It finally makes sense! In the process of writing this paper, three new humans entered the world. We would like to thank Juniper, Tommi, and baby sister for keeping us grounded in reality and giving us time at 2am to think about math.  The first author was supported by Deutsche Forschungsgemeinschaft (DFG), project number 45744154, Equivariant $K$-motives and Koszul duality.

\section{Whittaker modules} \label{sec: Whittaker modules}
\subsection{Definitions} \label{sec: definitions}

Let $\mf{g}$ be a semisimple complex Lie algebra and $\mf{b}$ a fixed Borel subalgebra of $\mf{g}$. Denote by $\mf{n} = [\mf{b}, \mf{b}]$ the nilpotent radical of $\mf{b}$, and $\mf{h}$ a Cartan subalgebra such that $\mf{b} = \mf{h} \oplus \mf{n}$. Set $\bar{\n}$ to be the nilpotent radical of the opposite Borel subalgebra $\bar{\mf{b}}$. 

Let $\Pi \subseteq \Sigma^+ \subseteq \Sigma \subseteq \mf{h}^*$ be the sets of simple and positive roots in the root system of $\mf{g}$ determined by our choice of $\mf{b}$, and let $(W, S)$ be the corresponding Coxeter system. Denote by $\mf{g}_\alpha$ the root space corresponding to $\alpha \in \Sigma$, so $\mf{n} = \bigoplus_{\alpha \in \Sigma^+} \mf{g}_\alpha$\footnote{Our convention that root spaces in $\mf{n}$ correspond to positive roots is the opposite to that in \cite{BBJantzen}, where $\mf{n}$ is comprised  of negative root spaces. Due to this choice, results in \cite{BBJantzen} which are stated for dominant $\lambda \in \mf{h}^*$ hold for antidominant $\lambda \in \mf{h}^*$ in this paper. This becomes relevant in Section \ref{sec: D-modules}.} Set $\rho \in \h^*$ to be the half-sum of positive roots. 

For $\alpha \in \Sigma$, denote by $\alpha^\vee \in \Sigma^\vee$ the corresponding coroot in the dual root system. We say that a weight $\lambda \in \mf{h}^*$ is {\em regular} if $\alpha^\vee(\lambda) \neq 0$ for any $\alpha \in \Sigma$, {\em integral} if $\alpha^\vee(\lambda) \in \Z$ for all $\alpha \in \Sigma$, and {\em antidominant} if $\alpha^\vee(\lambda) \not \in \mathbb{Z}_{>0}$.

Let $\ch\mf{n}$ be the set of Lie algebra homomorphisms $\eta: \mf{n} \rightarrow \C$. For any Lie algebra $\mf{a}$, denote by $U(\mf{a})$ its universal enveloping algebra and $Z(\mf{a})$ the center of $U(\mf{a})$. Any character $\eta \in \ch \mf{n}$ can be extended uniquely to an algebra homomorphism $\eta: U(\mf{n}) \rightarrow \C$ which we call by the same name. 

\begin{definition}
\label{def: Whittaker module}
Let $\eta \in \ch \mf{n}$. An {\em $\eta$-Whittaker vector} in a $\mf{g}$-module $M$ is a vector $m \in M$ such that $X \cdot m = \eta(X) m$ for all $X \in \mf{n}$. A {\em Whittaker module} is a $\mf{g}$-module which is cyclically generated by an $\eta$-Whittaker vector for some $\eta \in \ch{\mf{n}}$. 
\end{definition}

In \cite{CompositionSeries}, Mili\v{c}i\'{c}--Soergel introduced a category whose simple objects are the simple Whittaker modules. 

\begin{definition}
\label{def: Whittaker category}
The {\em Whittaker category} $\mc{N}$ is the category of finitely generated $U(\mf{g})$-modules which are $U(\mf{n})$-finite and $Z(\mf{g})$-finite. 
\end{definition}

The Whittaker category has a block decomposition given by the finiteness conditions in Definition \ref{def: Whittaker category}. For $\lambda \in \mf{h}^*$, set
\begin{equation}
\label{eq: inf char for g}
\chi_\lambda: Z(\mf{g}) \rightarrow \C; z \mapsto (\lambda - \rho)p(z),
\end{equation}
where 
\begin{equation}
\label{eq: HC hom}
p:Z(\mf{g}) \rightarrow U(\mf{h})
\end{equation}
is the Harish-Chandra homomorphism, defined as projection onto the first coordinate in the direct sum decomposition 
\begin{equation}
U(\mf{g}) = U(\mf{h}) \oplus(\bar{\n} U(\g) + U(\g) \n).     
\end{equation}
Under these conventions, $\chi_\lambda = \chi_\mu$ if and only if $\mu=w\lambda$ for some $w \in W$, where $w\lambda$ denotes the usual (not dot) action of $W$ on $\h^*$. Set $\theta = W \lambda$ to be the Weyl group orbit of $\lambda$ in $\mf{h}^*$. Then $J_\theta := \ker \chi_\lambda$ is a maximal ideal in $Z(\mf{g})$. Set $\mc{N}(\theta)$ and $\mc{N}(\hat{\theta})$ to be the full subcategories of $\mc{N}$ consisting of modules annihilated by $J_\theta$ and some power of $J_\theta$, respectively. The objects in  $\mc{N}(\theta)$  (resp. $\mc{N}(\hat{\theta})$) are the modules in $\mc{N}$ with infinitesimal character (resp. generalized infinitesimal character) $\chi_\lambda$. 

For a character $\eta \in \ch \mf{n}$, set $\mc{N}(\eta)$ to be the full subcategory of $\mc{N}$ consisting of modules $V$ satisfying 
\[
V=\{v \in V \mid (X - \eta(X))^k v = 0, X \in \mf{n}, \text{ for some }k \in \mathbb{N}\}. 
\]
Set $\mc{N}(\hat{\theta}, \eta) = \mc{N}(\hat{\theta}) \cap \mc{N}(\eta)$ and $\mc{N}(\theta, \eta) = \mc{N}(\theta) \cap \mc{N}(\eta)$. Clearly every simple Whittaker module is in some $\mc{N}(\theta, \eta)$. 

\begin{proposition}
\label{prop: block decomposition}
(\cite{CompositionSeries, TwistedSheaves}) The category $\mc{N}$ decomposes 
\[
\mc{N} = \bigoplus_{\theta \in W \backslash \mf{h}^*} \bigoplus_{\eta \in \ch{\mf{n}}} \mc{N}(\hat{\theta}, \eta).
\]
\end{proposition}

Let $N$ be the unipotent subgroup of $\Int \mf{g}$ such that $\Lie N = \mf{n}$. Given a module $V \in \mc{N}(\eta)$, we can twist by $-\eta$ to obtain a module $V \otimes \C_{-\eta}\in \mc{N}(0)$ on which the action of $\mf{n}$ is locally nilpotent, and hence integrates to an algebraic action of $N$. The natural isomorphism $V \xrightarrow{\sim} V \otimes \C_{-\eta}, v \mapsto v \otimes 1$ induces an algebraic action of $N$ on $V$ whose differential differs from the $\mf{n}$-action by $\eta$. This establishes that the categories $\mc{N}(\theta, \eta)$ can be viewed as ``twisted Harish-Chandra modules''. 
\begin{proposition}
\label{prop: twisted HC modules} (\cite[Lem. 2.3]{TwistedSheaves})
There is an equivalence of categories 
\[
\mc{N}(\theta, \eta) = \mc{M}_{f.g.}(U_\theta, N, \eta),
\]
where $U_\theta$ is the quotient of $U(\mf{g})$ by the ideal generated by $J_\theta$ and $\mc{M}_{f.g.}(U_\theta,N, \eta)$ is the category of finitely generated $U_\theta$-modules admitting an algebraic action of $N$ whose differential differs from the $\mf{n}$-action obtained by restriction exactly by $\eta$. 
\end{proposition}

\subsection{Standard/Costandard modules and contravariant pairings} \label{sec: standard/costandard modules and contravariant pairings}

The categories $\mc{N}(\theta, \eta)$ have a natural set of standard, costandard, and simple objects, and with respect to these objects, they are highest weight categories \cite[Cor. 7.4]{BRpairings}. In this section, we review the construction of standard, costandard, and simple objects, and recall the classification of contravariant pairings between standard Whittaker modules and Verma modules introduced in \cite{BRpairings}. 

 A character $\eta \in \ch \mf{n}$ determines a subset of simple roots 
\begin{equation}
    \label{eq: pi_eta}
    \Pi_\eta:= \{ \alpha \in \Pi \mid \eta|_{\mf{g}_\alpha} \neq 0 \}.
\end{equation}
Let $\Sigma_\eta$ be the root system generated by $\Pi_\eta$ and $W_\eta$ the Weyl group of $\Sigma_\eta$. Denote by
\begin{equation}
    \label{eq: all the eta subalgebras}
    \mf{l}_\eta = \mf{h} \oplus \bigoplus_{\alpha \in \Sigma_\eta} \mf{g}_\alpha, \hspace{2mm} \mf{n}_\eta = \mf{l}_\eta \cap \mf{n}, \hspace{2mm} \bar{\mf{n}}_\eta = \mf{l}_\eta \cap \bar{\mf{n}}, \hspace{2mm} \mf{n}^\eta = \bigoplus_{\alpha \in \Sigma^+ \backslash \Sigma_\eta^+}\mf{g}_\alpha, \hspace{2mm} \mf{p}_\eta = \mf{l}_\eta \oplus \mf{n}^\eta. 
\end{equation}
The Levi subalgebra $\mf{l}_\eta$ is reductive with center 
\begin{equation}
    \label{eq: center of levi}
    \mf{h}^\eta = \{H \in \mf{h} \mid \alpha(H) = 0, \alpha \in \Pi_\eta\}. 
\end{equation}
Let 
\begin{equation}
    \label{eq: HC hom of l}
p_\eta: Z(\mf{l}_\eta) \rightarrow U(\mf{h})
\end{equation}
be the Harish-Chandra homomorphism of $U(\mf{l}_\eta)$, defined as projection onto the first coordinate in the direct sum decomposition 
\begin{equation}
    U(\mf{l}_\eta) = U(\h) \oplus (\bar{\mf{n}}_\eta U(\mf{l}_\eta) + U(\mf{l}_\eta) \mf{n}_\eta).
\end{equation} 
For $\lambda \in \mf{h}^*$, let 
\begin{equation}
    \label{eq: lambda inf character}
    \chi_\eta^\lambda : Z(\mf{l}_\eta) \rightarrow \C; z \mapsto (\lambda - \rho_\eta) p_\eta(z)
    \end{equation}
    be the corresponding infinitesimal character. Here $\rho_\eta = \frac{1}{2} \sum_{\alpha \in \Sigma_\eta^+} \alpha$ is the half-sum of positive roots in $\Sigma_\eta$. We have $\chi_\eta^\lambda = \chi_\eta^\mu$ if and only if $\mu = w \lambda$ for some $w \in W_\eta$. 

From the data $(\lambda, \eta) \in \mf{h}^* \times \ch \mf{n}$, we construct a $U(\mf{l}_\eta)$-module 
\begin{equation}
    \label{eq: nondegenerate standard}
    Y(\lambda, \eta) := U(\mf{l}_\eta) \otimes_{Z(\mf{l}_\eta) \otimes U(\mf{n}_\eta)} \C_{\chi_\eta^\lambda, \eta}.
\end{equation}
Here $\C_{\chi_\eta^\lambda}$ is the one-dimensional $Z(\mf{l}_\eta) \otimes U(\mf{n}_\eta)$-module with action 
\begin{equation}
    \label{eq: 1d module in standard}
    z \otimes u \cdot v = \chi_\eta^\lambda(z)\eta(u)v
\end{equation}
for $z \in Z(\mf{l}_\eta), u \in U(\mf{n}_\eta), v \in \C$. By construction, $Y(\lambda, \eta) \simeq Y(\mu, \eta)$ if and only if $\mu = w \lambda$ for some $w \in W_\eta$. For any $\lambda \in \mf{h}^*$, $Y(\lambda, \eta)$ is an irreducible $U(\mf{l}_\eta)$-module \cite[Prop. 2.3]{McDowell}. 
\begin{definition}
\label{def: standard Whittaker module}
The {\em standard Whittaker module} associated to $(\lambda, \eta) \in \mf{h}^* \times \ch \mf{n}$ is the $U(\mf{g})$-module
\[
M(\lambda, \eta):= U(\mf{g}) \otimes_{U(\mf{p}_\eta)} Y(\lambda - \rho + \rho_\eta, \eta).
\]
Here we view $Y(\lambda - \rho + \rho_\eta)$ as a $U(\mf{p}_\eta)$-module by letting $\mf{n}^\eta$ act trivially. 
\end{definition}
The properties of standard Whittaker modules were established in \cite{McDowell}. We collect them in the following proposition for future reference. 
\begin{proposition}
\label{prop: properties of standard Whittakers}
\begin{enumerate}[label=(\roman*)]
    \item If $\theta=W\lambda$, then $M(\lambda, \eta) \in \mc{N}(\theta, \eta)$.
    \item Two standard modules $M(\lambda, \eta)$ and $M(\mu, \eta)$ are isomorphic if and only if $\mu \in W_\eta \lambda$. 
    \item The standard module $M(\lambda, \eta)$ is a Whittaker module generated by the $\eta$-Whittaker vector $1 \otimes 1 \otimes 1$. 
    \item For $\lambda \in \mf{h}^*$, denote by $\lambda_\eta$ the restriction of $\lambda$ to $\mf{h}^\eta$ (\ref{eq: center of levi}). There is a natural partial order on $\mf{h}^{\eta*} = \Hom_\C(\mf{h}^\eta, \C)$ obtained from the partial order on $\mf{h}^*$. The Lie algebra $\mf{h}^\eta$ acts semisimply on $M(\lambda, \eta)$, and with respect to this action, $M(\lambda, \eta)$ decomposes into $\mf{h}^\eta$-weight spaces: 
    \[
    M(\lambda, \eta) = \bigoplus_{\nu_\eta \leq \lambda_\eta-\rho_\eta} M(\lambda, \eta)_{\nu_\eta}.
    \]
    Each $\mf{h}^\eta$-weight space $M(\lambda, \eta)_{\nu_\eta}$ is $U(\mf{l}_\eta)$-stable, and as $U(\mf{l}_\eta)$-modules, 
    \[
    M(\lambda, \eta)_{\nu_\eta} = U(\bar{\mf{n}}^\eta)_{\mu_\eta} \otimes_\C Y(\lambda - \rho + \rho_\eta, \eta)
    \]
    for some $\mu_\eta \leq 0$ in $\mf{h}^{\eta *}$. Here $\bar{\mf{n}}^\eta = \bigoplus_{\alpha \in - (\Sigma^+ \backslash \Sigma_\eta^+)} \mf{g}_\alpha$.
    \item The standard module $M(\lambda, \eta)$ has a unique irreducible quotient $L(\lambda, \eta)$, which is also an $\eta$-Whittaker module. All simple Whittaker modules are isomorphic to $L(\lambda, \eta)$ for some $(\lambda, \eta) \in \mf{h}^* \times \ch \mf{n}$, and the collection of $L(\lambda, \eta)$ forms a complete set of simple modules in $\mc{N}$. 
    \item If $\eta = 0$, then $M(\lambda, 0)$ is a Verma module of highest weight $\lambda - \rho$. 
\end{enumerate}
\end{proposition}

In \cite{BRpairings}, costandard modules in category $\mc{N}$ were constructed. For a $U(\mf{g})$-module $V$, denote by 
\begin{equation}
    \label{eq: twisted U(n) finite vectors}
    (V)_\eta: = \{ v \in V \mid (X - \eta(X))^kv = 0, X \in \mf{n} \text{ for some } k \in \mathbb{N}\}
\end{equation}
the set of $\eta$-twisted $\mf{n}$-finite vectors in $V$. For any $U(\mf{g})$-module, $(V)_\eta$ is $U(\mf{g})$-stable \cite[Lem. 6.1]{BRpairings}. For a $U(\mf{g})$-module $V$, we denote by $V^* = \Hom_\C(V, \C)$ the full linear dual of $V$, which we consider as a $U(\mf{g})$-module via the action $u \cdot f(-) = f(\tau(u) \cdot - )$, where $\tau$ is the transpose antiautomorphism of $\mf{g}$. 
\begin{definition}
\label{def: costandard modules}
The {\em costandard module} in $\mc{N}$ associated to $(\lambda, \eta) \in \mf{h}^* \times \ch \mf{n}$ is
\[
M^\vee(\lambda, \eta) := (M(\lambda)^*)_\eta,
\]
where $M(\lambda)$ is the Verma module of highest weight $\lambda - \rho$. 
\end{definition}

Unlike standard modules, costandard modules are not Whittaker modules, though they are in category $\mc{N}$. Standard and costandard modules in $\mc{N}$ are related via contravariant pairings. 
\begin{definition}
\label{def: contravariant pairing}
A {\em contravariant pairing} between the standard Whittaker module $M(\lambda, \eta)$ and the Verma module $M(\mu)$ is a bilinear pairing 
\[
( \cdot, \cdot ): M(\lambda, \eta) \times M(\mu) \rightarrow \C
\]
such that $( ux, y ) = ( x, \tau(u)y )$ for all $u \in U(\mf{g})$ and $y \in M(\mu)$.
\end{definition}
\begin{theorem}
\label{thm: contravariant pairings}
\cite[Thm. 5.2]{BRpairings} If $\lambda \in \mf{h}^*$ is regular, there exists a contravariant pairing between $M(\lambda, \eta)$ and $M(\mu)$ if and only if $\mu \in W_\eta \lambda$. If a contravariant pairing between $M(\lambda, \eta)$ and $M(\mu)$ exists, it is unique up to scaling.
\end{theorem}

Contravariant pairings on standard Whittaker modules are a generalization of contravariant forms on Verma modules, and they provide the canonical map between standard and costandard objects in category $\mc{N}$. The properties of costandard Whittaker modules were developed in \cite{BRpairings}, and we list them in the following proposition for future reference. 
\begin{proposition}
\label{prop: properties of costandards}
\begin{enumerate}[label=(\roman*)]
    \item Two costandard modules $M^\vee(\lambda, \eta)$ and $M^\vee(\mu, \eta)$ are isomorphic if and only if $\mu \in W_\eta \lambda$.
    \item The costandard module $M^\vee(\lambda, \eta)$ is an object in the category $\mc{N}(\theta, \eta)$, where $\theta = W \lambda$.
    \item The unique contravariant pairing $\langle \cdot, \cdot \rangle: M(\lambda, \eta) \times M(\lambda) \rightarrow \C$ (Theorem \ref{thm: contravariant pairings}, \eqref{eq: unique pairing}) induces a canonical $\mf{g}$-module homomorphism $M(\lambda, \eta) \rightarrow M^\vee(\lambda, \eta)$. 
    \item In the Grothendieck group $K\mc{N}(\theta, \eta)$,  $[M(\lambda, \eta)] = [M^\vee(\lambda, \eta)]$, so the standard and costandard module corresponding to a given parameter have the same composition factors.
    \item The costandard module $M^\vee(\lambda, \eta)$ has a unique irreducible submodule which is isomorphic to $L(\lambda, \eta)$. 
\end{enumerate}
\end{proposition}
\subsection{Deformed modules} \label{sec: deformed modules}

To define the Jantzen filtrations algebraically, we need to introduce deformed versions of the objects above. To begin, we recall the construction of deformed category $\mc{O}$, following \cite{Soergel}.

\subsubsection{Deformed category $\mc{O}$}
\label{sec: deformed category O}

Let $T$ be a commutative unital $\C$-algebra with a distinguished algebra homomorphism  
\begin{equation}
    \label{eq: symmetric algebra to T}
    U(\h)=S(\h) = \mathcal{O}(\h^*) \xrightarrow{\varphi} T.
\end{equation}
Let $\bimod{\mf{g}}{T}$ denote the category of bimodules over $\g$ and $T.$ For $\lambda \in \mf{h}^*$ and $M \in \gmodT$, define the {\em $\lambda$-deformed weight space} to be
\begin{equation}
    \label{eq: deformed weight space}
    M^\lambda := \{m \in M \mid (H - \lambda(H))m = m \varphi(H) \text{ for all } H \in \mf{h}\}.
\end{equation}

\begin{definition}
\label{def: deformed category O}
{\em Deformed category $\mc{O}$} is the full subcategory $\mc{O}_T \subseteq \gmodT$ of locally $\mf{n}$-finite modules in $\gmodT$ with the property that $M = \bigoplus_\lambda M^\lambda$. 
\end{definition}

We are especially interested in two classes of objects in $\mc{O}_T$: deformed Verma modules and deformed dual Verma modules. 

\begin{definition}
\label{def: deformed Verma module}
For $\lambda \in \mf{h}^*$, the corresponding {\em deformed Verma module} is the $(\mf{g},T)$-bimodule 
\[
    M_T(\lambda):= U(\mf{g}) \otimes_{U(\mf{b})}T_{\lambda - \rho},
\]
where $\g$ acts by left multiplication on the first tensor factor and $T$ acts on the final tensor factor by right multiplication. The $U(\mf{b})$-module structure on $T_{\lambda - \rho}$ is given by extending the $\mf{h}$-action 
\[
H \cdot t = (\lambda - \rho + \varphi)(H) t 
\]
for $H \in \mf{h}$ and $t \in T$ trivially to $\mf{b}$. 
\end{definition}

\begin{definition}
\label{def: deformed dual Verma}
The {\em deformed dual Verma module} $M^\vee_T(\lambda)$ corresponding to $\lambda \in \mf{h}^*$ is the sum of the deformed weight spaces in 
\[
\mbox{ind}_{\bar{\mf{b}}}^\mf{g}(T_{\lambda - \rho}) = \Hom_{U(\bar{\mf{b}})}(U(\mf{g}), T_{\lambda - \rho}).
\]
Here $\Hom_{U(\bar{\mf{b}})}(U(\mf{g}), T_{\lambda - \rho})$ refers to homomorphisms of left $U(\overline{\mf{b}})$-modules. It is given the structure of a left $U(\mf{g})$-module   via the right action of $U(\mf{g})$ on itself. 
\end{definition}

\begin{remark}
\label{rem: different definitions of dual Verma modules}
As $(\mf{g},T)$-bimodules, 
\begin{equation}
    \label{eq: two dual vermas}
    \Hom_{U(\bar{\mf{b}})}(U(\mf{g}), T_{\lambda - \rho}) \simeq \Hom_T(M_T(\lambda), T)
\end{equation}
where $\Hom_T(M_T(\lambda), T)$ denotes the set of $T$-module homomorphisms from $M_T(\lambda)$ to $T$. The set $\Hom_T(M_T(\lambda), T)$ is given the structure of a left $U(\mf{g})$-module via the action
\[
u \cdot f(v \otimes t) = f(\tau(u)v \otimes t),
\]
for $u,v \in U(\mf{g})$, $f \in \Hom_T(M_T(\lambda), T)$, and $t \in T_{\lambda - \rho}$. The isomorphism \eqref{eq: two dual vermas} comes from the left $U(\mf{g})$-module isomorphism \footnote{Here superscripts on $\Hom$ indicate whether we are considering left or right modules, and the left $U(\mf{g})$-module structure on $\Hom^r_{U(\mf{b})}(U(\mf{g}), T_{\lambda - \rho})$ comes from the right action of $U(\mf{g})$ on itself via left multiplication by the transpose.} 
\[
\Hom_{U(\bar{\mf{b}})}^l(U(\mf{g}), T_{\lambda - \rho}) \simeq \Hom^r_{U(\mf{b})}(U(\mf{g}), T_{\lambda - \rho})
\]
and an application of tensor-hom adjunction for right modules. 
\end{remark}
 The bimodules $M_T(\lambda)$ and $M^\vee_T(\lambda)$ are objects in $\mc{O}_T$. The deformed weight spaces of $M_T(\lambda)$ and $M^\vee_T(\lambda)$ are free finite-rank $T$-modules, and in both modules, for $T \neq 0$, the rank of the deformed $(\lambda - \rho - \nu)$-weight space is $\dim_\C U(\mf{n})^\nu$. 

\begin{remark}
\label{rem: recovering category O}
If we specialize to the case $T=\C$ and let $\varphi$ be evaluation at zero, then $\mc{O}_T$ is a version of Bernstein--Gelfand--Gelfand's category $\mc{O}$ which is missing some of the usual finiteness conditions. In this case, the $M_T(\lambda)$ and $M_T^\vee(\lambda)$ are the Verma module and dual Verma module of highest weight $\lambda - \rho$. 
\end{remark}

\subsubsection{Deformed Whittaker modules}

Within the category of $(\mf{g},T)$-bimodules, we can also define deformed standard and costandard Whittaker modules. 

\begin{definition}
\label{def: deformed standard Whittaker}
For a pair $(\lambda, \eta) \in \mf{h}^* \times \ch{\mf{n}}$, we define a {\em deformed standard Whittaker module} to be the $(\mf{g},T)$-bimodule
\[
M_T(\lambda, \eta) = U(\mf{g}) \otimes_{U(\mf{p}_\eta)} U(\mf{l}_\eta) \otimes_{Z(\mf{l}_\eta) \otimes U(\mf{n}_\eta)} T_{\lambda - \rho+\rho_\eta,\eta}
\]
where $\g$ acts by left multiplication on the first tensor factor and $T$ acts by right multiplication on the final tensor factor. The $Z(\mf{l}_\eta) \otimes U(\mf{n}_\eta)$-module structure on $T_{\lambda - \rho+\rho_\eta}$ is given as follows: for $z \otimes u \in Z(\mf{l}_\eta) \otimes U(\mf{n}_\eta)$ and $ t \in T_{\lambda - \rho+\rho_\eta}$, 
\[
z \otimes u \cdot t = (\chi^{\lambda-\rho+\rho_\eta}_\eta + \varphi \circ p_\eta)(z) \eta(u) t,
\]
where $p_\eta$ is the Harish-Chandra homomorphism of $\mf{l}_\eta$ \eqref{eq: HC hom of l}, $\chi_\eta^{\lambda-\rho+\rho_\eta}$ is the infinitesimal character of $Z(\mf{l}_\eta)$ corresponding to $\lambda-\rho+\rho_\eta$ \eqref{eq: lambda inf character}, and $\varphi$ is as in \eqref{eq: symmetric algebra to T}. 

\end{definition}
Because $\chi_\eta^{\lambda + \rho - \rho_\eta}  = \chi_\eta^{w \lambda + \rho - \rho_\eta}$ for all $w \in W_\eta$, we see immediately from the definition that 
\begin{equation}
    \label{eq: standard Whittakers in an orbit are isomorphic}
    M_T(\lambda, \eta) \simeq M_T(w \lambda, \eta)
\end{equation}
for all $w \in W_\eta$. 

Recall that for a $U(\mf{g})$-module $V$, $(V)_\eta$ denotes the $U(\mf{g})$-submodule of $\eta$-twisted $\mf{n}$-finite vectors in $V$ \eqref{eq: twisted U(n) finite vectors}. 
\begin{definition}
\label{def: deformed costandard Whittaker} For a pair $(\lambda, \eta) \in \mf{h}^* \times \ch{\mf{n}}$, the corresponding {\em deformed costandard Whittaker module} is the $(\g, T)$-bimodule
\[
M^\vee_T(\lambda, \eta) = \left( \Hom_{U(\bar{\mf{b}})}(U(\mf{g}),  T_{\lambda - \rho}) \right)_\eta.
\]
Here the $U(\mf{g})$-module structure on $\Hom_{U(\bar{\mf{b}})}(U(\mf{g}), T_{\lambda - \rho})$ is given by right multiplication of $U(\mf{g})$ on itself. 
\end{definition}

\subsection{Two algebraic Jantzen filtrations}
\label{sec: two algebraic Jantzen filtrations}

\subsubsection{The Jantzen filtration of a Verma module}
\label{sec: the Jantzen filtration of a Verma module}
There is an isomorphism
\label{sec: The Jantzen filtration of a Verma module}
\begin{equation}
    \label{eq: homs between vermas and dual}
    \Hom_{\mc{O}_T}(M_T(\lambda), M_T^\vee(\lambda)) \xrightarrow{\sim} T
\end{equation}
induced from restriction to the deformed $(\lambda - \rho)$-weight space together with the canonical identifications $M_T(\lambda)^{\lambda - \rho} \xrightarrow{\sim} T$ and $M^\vee_T(\lambda)^{\lambda - \rho} \xrightarrow{\sim} T$ \cite[Prop. 2.12]{Soergel}. Under this isomorphism, $1 \in T$ distinguishes a canonical $(\mf{g}, T)$-bimodule homomorphism 
\begin{equation}
    \label{eq: canonical map of vermas}
\psi: M_T(\lambda) \rightarrow M_T^\vee(\lambda).
\end{equation}

To define the Jantzen filtration of a deformed Verma module, we must specialize slightly. Let $A$ be a local $\C$-algebra and $\varphi: U(\mf{h}) = \mc{O}(\mf{h}^*) \rightarrow A$ a distinguished ring homomorphism such that $\varphi (\mf{h}) \subseteq \mf{m}$, where $\mf{m}$ is the unique maximal ideal of $A$. 

 For any right $A$-module $M$, there is a ascending $A$-module filtration $M_{-i} = M \mf{m}^i$ with associated graded $\mathrm{gr}_iM = M_i/M_{i-1}$. In particular, there is a reduction map 
 \begin{equation}
     \label{eq: reduction map}
     \mathrm{red}: M \rightarrow \mathrm{gr}_0 M =  M/M \mf{m}, m \mapsto \overline{m}.
 \end{equation}
 Applying this to the $A$-module $M_A(\lambda)$, we see that the filtration layers are $\mf{g}$-stable, so we obtain
 a surjective $\mf{g}$-module homomorphism 
 \begin{equation}
     \label{eq: reduction map verma}
     \mathrm{red}: M_A(\lambda) \rightarrow \mathrm{gr}_0M_A(\lambda) = M(\lambda).
 \end{equation}
Similarly, the reduction map for $M_A^\vee(\lambda)$ gives a surjective $\mf{g}$-module homomorphism from $M^\vee_A(\lambda)$ to the dual Verma module $M^\vee(\lambda)$. 
Pulling back the filtration above along the canonical map $\psi$ \eqref{eq: canonical map of vermas}, we obtain a $(\mf{g},A)$-bimodule filtration of $M_A(\lambda)$.
\begin{definition}
\label{def: Jantzen filtration of Verma} The {\em Jantzen filtration of $M_A(\lambda)$} is the $(\mf{g},A)$-bimodule filtration
\[
M_A(\lambda)_{-i} : = \{ m \in M_A(\lambda) \mid \psi(m) \in M^\vee_A(\lambda) \mf{m}^i \},
\]
where $\psi$ is the canonical map in \eqref{eq: canonical map of vermas}. By applying the map $\mathrm{red}$ \eqref{eq: reduction map verma} to the filtration layers, we obtain in this way a filtration $M(\lambda)_\bullet$ of $M(\lambda)$, which we refer to as the {\em Jantzen filtration of $M(\lambda)$}. 
\end{definition}

\begin{remark}
    It is convenient for us to view the Jantzen filtration as an ascending filtration, as this aligns more naturally with the geometric filtrations occurring in Section \ref{sec: D-modules}. Because of this, our definition of the Jantzen filtration differs slightly from that in \cite{Jantzen}, where it is defined as a descending filtration. 
\end{remark}

This filtration was introduced using contravariant forms in work of Jantzen \cite{Jantzen}. In the case where the deformation direction is $\rho$, the filtration satisfies surprising functoriality and semisimplicity properties. More precisely, we further specialize to the following setting.

Set $T = \mc{O}(\C \rho)$ to be the ring of regular functions on the line $\C \rho \subset \mf{h}^*$. A choice of a linear functional $v:\C\rho \rightarrow \C$ gives an isomorphism $T \cong \C[v]$. We fix such an identification. Set 
\begin{equation}
    \label{eq: completion of T}
    A = T_{(v)}
\end{equation}
to be the local $\C$-algebra obtained from $T$ by inverting all polynomials not divisible by  $v$, and 
\begin{equation}
    \label{eq: phi}
    \varphi: \mc{O}(\mf{h}^*) \rightarrow A
\end{equation}
to be the composition of the restriction map $\mc{O}(\mf{h}^*) \rightarrow T$ with inclusion $T \hookrightarrow A$. 

\begin{theorem}
\label{thm: Jantzen conjecture}
 \cite[Cor. 5.3.4, Cor. 5.3.1]{BBJantzen} {\em (The Jantzen Conjectures for Verma modules)} Let $A$ and $T$ be as in \eqref{eq: completion of T}.
 \begin{enumerate}
     \item If $\iota: M(\mu) \hookrightarrow M(\lambda)$ is an embedding of Verma modules, then the filtration on $\iota(M(\mu))$ induced from the Jantzen filtration on $M(\lambda)$ agrees with the Jantzen filtration on $M(\mu)$, up to a shift. 
     \item The filtration layers $M(\lambda)_i/M(\lambda)_{i-1}$ of the Jantzen filtration of a Verma module $M(\lambda)$ are semisimple for all $i$.
 \end{enumerate} 
\end{theorem}

The strategy of Beilinson--Bernstein's proof of Theorem \ref{thm: Jantzen conjecture} is to relate the Jantzen filtration to the weight filtration of the corresponding $\mc D$-module. We will follow a similar course for Whittaker modules in Section \ref{sec: Proof of the Jantzen Conjecture}.

\subsubsection{The Jantzen filtration of a standard Whittaker module}
\label{sec: The Jantzen filtration of a standard Whittaker module}

We can generalize the classical construction above to Whittaker modules using the canonical map between standard and costandard Whittaker modules (Proposition \ref{prop: properties of costandards}(iii)). Our first step is to establish the existence of a canonical map between standard and costandard Whittaker modules on the deformed level. To do so, we use $A$-valued contravariant pairings. 

We return to the setting where $A$ is a local $\C$-algebra equipped with a distinguished algebra homomorphism $\varphi:U(\mf{h}) \rightarrow A$ which maps $\mf{h}$ into the maximal ideal $\mf{m}$.  

\begin{definition}
\label{def: A-contravariant pairing} An {\em $A$-contravariant pairing} between $(\mf{g},A)$-bimodules $V$ and $W$ is an $A$-valued $A$-bilinear pairing
\[
( \cdot, \cdot )_A: V \times W \rightarrow A
\]
such that for all $u \in U(\mf{g})$, $v \in V$ and $w \in W$, $( uv, w )_A = ( v, \tau(u) w )_A$, where $\tau:U(\mf{g}) \rightarrow U(\mf{g})$ is the transpose antiautomorphism. 
\end{definition}
Clearly the set $\Psi_A(V,W)$ of $A$-contravariant pairings between $V$ and $W$ is a right $A$-module. We are interested in the $A$-module
\begin{equation}
    \label{eq: psi_a}
    \Psi_A:= \Psi_A(M_A(\lambda, \eta), M_A(\lambda))
\end{equation}
of $A$-contravariant pairings between deformed standard Whittaker modules and deformed Verma modules, which is closely related to the $\C$-vector space
\begin{equation}
    \label{eq: psi}
    \Psi:= \{ \text{contravariant pairings between $M(\lambda, \eta)$ and $M(\lambda)$} \}
\end{equation}
studied in \cite{BRpairings}. By Theorem \ref{thm: contravariant pairings}, $\Psi$ is $1$-dimensional. 

Let $w = 1 \otimes 1 \otimes 1 \in M(\lambda, \eta)$ be the generating Whittaker vector and $v = 1 \otimes 1 \in M(\lambda)$ be the generating highest weight vector. Denote by 
\begin{equation}
    \label{eq: unique pairing}
    \langle \cdot, \cdot \rangle: M(\lambda, \eta) \otimes M(\lambda) \rightarrow \C
\end{equation}
the unique contravariant pairing such that $\langle w, v \rangle = 1$ (Theorem \ref{thm: contravariant pairings}). Set $\underline{w} = 1 \otimes 1 \otimes 1 \in M_A(\lambda, \eta)$ and $\underline{v} = 1 \otimes 1 \in M_A(\lambda)$. Clearly $\underline{w}$ (resp. $\underline{v}$) generates $M_A(\lambda, \eta)$ (resp. $M_A(\lambda)$) as a $(\mf{g}, A)$-bimodule. We define a $A$-bilinear pairing
\begin{equation}
    \label{eq: unique A pairing}
    \langle \cdot, \cdot \rangle_A: M_A(\lambda, \eta) \otimes M_A(\lambda) \rightarrow A
\end{equation}
by $\langle u\underline{w} a, u' \underline{v}b \rangle_A := \langle uw, u'v \rangle ab$ for $u, u' \in U(\mf{g}), a, b \in A$. 

\begin{proposition}
\label{prop: well-defined}
The map $\langle \cdot, \cdot \rangle_A$ in \eqref{eq: unique A pairing} is an $A$-contravariant pairing of $M_A(\lambda, \eta)$ and $M_A(\lambda)$. 
\end{proposition}

\begin{proof}
By construction, $\langle \cdot, \cdot \rangle_A$ is $A$-bilinear and contravariant. It remains to show that it is well-defined. 

A vector $x \in M_A(\lambda, \eta)$ can be represented as $u \underline{w} a$ for $u \in U(\mf{g}), a \in A$ in many different ways, and similarly for $y =u'\underline{v} b \in M_A(\lambda)$. We will show that our definition is independent of such choices. Because $\langle \cdot, \cdot \rangle_A$ is defined to be $A$-bilinear and $M_A(\lambda, \eta)$ and $M_A(\lambda)$ are free as $A$-modules, it suffices to consider the case when $a=b=1$. 

Let $u, u', z, z' \in U(\mf{g})$ be such that 
\[
x = u \underline{w} = z \underline{w} \text{ and } y = u' \underline{v} = z' \underline{v}. 
\]
Then $u - z \in \Ann_{U(\mf{g})} \underline{w}$ and $u'-z' \in \Ann_{U(\mf{g})} \underline{v}$. In fact, $u - z \in \Ann_{U(\mf{g})} w$ and $u'-z' \in \Ann_{U(\mf{g})} v$, because the reduction map $\mathrm{red}: M_A \rightarrow M$ \eqref{eq: reduction map} for $M = M(\lambda, \eta)$ or $M = M(\lambda)$ is a $U(\mf{g})$-module homomorphism. This implies that 
\[
\langle uw, u'v \rangle = \langle zw, z'v \rangle,
\]
because $\langle \cdot, \cdot \rangle$ is well defined, and hence 
\[
\langle u\underline{w}, u'\underline{v} \rangle_A = \langle z\underline{w}, z'\underline{v} \rangle_A. \qedhere
\]
\end{proof}

We have constructed one element $\langle \cdot, \cdot \rangle_A \in \Psi_A$. For any $a \neq 1 \in A$, $\langle \cdot , \cdot \rangle_A a$ is another (distinct) element of $\Psi_A$, so $A \subseteq \Psi_A$. We claim that there are no other $A$-contravariant pairings. 
\begin{proposition}
\label{prop: Psi_A is rank 1}
As a right $A$-module, $\Psi_A \simeq A$. 
\end{proposition}

\begin{proof}
Let $( \cdot, \cdot )_A $ be an $A$-contravariant pairing. From $( \cdot, \cdot )_A$, we can define a contravariant pairing
\[
( \cdot, \cdot ): M(\lambda, \eta) \times M(\lambda) \rightarrow \C
\]
by $( uw, u'v ) := \mathrm{red} ( u \overline{w}, u' \overline{v} ) _A$, where $\mathrm{red}: A \rightarrow \C$ is the reduction map \eqref{eq: reduction map} for the left $A$-module $A$. This defines a map 
\begin{equation}
    \label{eq: g}
    g: \Psi_A \rightarrow \Psi; ( \cdot, \cdot )_A \mapsto ( \cdot, \cdot ). 
\end{equation}

By the construction above Proposition \ref{prop: well-defined}, we also have a map
\begin{equation}
    \label{eq: f}
    f: \Psi \rightarrow \Psi_A,
\end{equation}
with $f(\langle \cdot, \cdot \rangle) = \langle \cdot, \cdot \rangle_A$. (Here $\langle \cdot, \cdot \rangle$ and $\langle \cdot, \cdot \rangle_A$ are the specific pairings defines in \eqref{eq: unique pairing} and \eqref{eq: unique A pairing}.) 

The maps $f$ and $g$ are not inverse to one another, but it is easy to check that for any $A$-contravariant pairing $( \cdot, \cdot )_A$, there exists $a \in A$ such that 
\begin{equation}
    \label{eq: multiple}
    f \circ g ( ( \cdot, \cdot )_A) a = ( \cdot, \cdot )_A. 
\end{equation}
Now, by Theorem \ref{thm: contravariant pairings}, 
\[
g(( \cdot, \cdot )_A) = ( \cdot, \cdot ) = c \langle \cdot, \cdot \rangle
\]
for some $c \in \C$, where $\langle \cdot, \cdot \rangle$ is the unique pairing in \eqref{eq: unique pairing}.  From this we see that for an $A$-contravariant pairing $( \cdot, \cdot )_A$, we have on one hand,
\[
f \circ g (( \cdot, \cdot )_A)a = ( \cdot, \cdot )_A
\]
for some $a \in A$, and on the other hand, 
\[
f \circ g ( (\cdot, \cdot )_A) = f (c \langle \cdot, \cdot \rangle ) = c\langle \cdot, \cdot \rangle_A. 
\]
Hence $( \cdot, \cdot )_A = c \langle \cdot, \cdot \rangle_A a$ is an $A$-multiple of $\langle \cdot , \cdot \rangle_A$. 
\end{proof}

We can use Proposition \ref{prop: Psi_A is rank 1} to construct a canonical map between deformed standard and costandard Whittaker modules. Recall that by Remark \ref{rem: different definitions of dual Verma modules}, we can view a deformed costandard Whittaker module (Definition \ref{def: deformed costandard Whittaker}) as a submodule of the full $A$-linear dual of a deformed Verma module:
\[
M^\vee_A(\lambda, \eta) \simeq \left( \Hom_A(M_A(\lambda), A) \right)_\eta \subseteq \Hom_A(M_A(\lambda), A).
\]

There is a canonical $A$-contravariant form $\langle \cdot, \cdot \rangle_A \in \Psi_A$ corresponding to $1 \in A$ with $\langle \underline w, \underline v \rangle_A=1$ (Proposition \ref{prop: Psi_A is rank 1}). This induces a $(\mf{g}, A)$-bimodule morphism \begin{equation}
    \label{eq: canonical map deformed Whittaker}
    \psi_\eta: M_A(\lambda, \eta) \rightarrow \Hom_A(M_A(\lambda), A); m \mapsto \langle m, \cdot \rangle_A. 
\end{equation}
Because $\psi_\eta$ is a $\mf{g}$-module homomorphism and $(M_A(\lambda, \eta))_\eta = M(\lambda, \eta)$, the image of $\psi_\eta$ lands in $M_A^\vee(\lambda, \eta) = \left( \Hom_A(M_A(\lambda), A) \right)_\eta$. 

Now we can define a Jantzen filtration of $M_A(\lambda, \eta)$ analogously to Definition \ref{def: Jantzen filtration of Verma} by pulling back the natural filtration of $M^\vee_A(\lambda, \eta)$ along $\psi_\eta$. 

\begin{definition}
\label{def: Jantzen filtration of a standard Whittaker}
The {\em Jantzen filtration of $M_A(\lambda, \eta)$} is the $(\mf{g}, A)$-bimodule filtration 
\[
M_A(\lambda, \eta)_{-i}:= \{ m \in M_A(\lambda, \eta) \mid \psi_\eta(m) \in M_A^\vee(\lambda, \eta) \mf{m}^i\},
\]
where $\mf{m}$ is the unique maximal ideal of $A$. By applying the reduction map \ref{eq: reduction map} to the filtration layers, we obtain a filtration $M(\lambda, \eta)_\bullet$ of $M(\lambda, \eta)$ which we refer to as the {\em Jantzen filtration of $M(\lambda, \eta)$. }
\end{definition}

The Jantzen filtration of a standard Whittaker module satisfies the same hereditary and semisimplicity properties as the Jantzen filtration of a Verma module. This is our main result. 

\begin{theorem}
\label{thm: Jantzen conjecture Whittaker} {\em (The Jantzen conjectures for Whittaker modules)} Let $A$ be the local $\C$-algebra in \eqref{eq: completion of T} determined by $\rho$. 
\begin{enumerate}
    \item If $\iota: M(\mu, \eta) \hookrightarrow M(\lambda, \eta)$ is an embedding of standard Whittaker modules, then the filtration on $\iota(M(\mu, \eta))$ induced from the Jantzen filtration on $M(\lambda, \eta)$ agrees with the Jantzen filtration on $M(\mu, 
    \eta)$, up to a shift. 
    \item Then filtration layers $M(\lambda, \eta)_i/M(\lambda, \eta)_{i-1}$ of the Jantzen filtration of a standard Whittaker module $M(\lambda, \eta)$ are semisimple for all $i$.
\end{enumerate}
\end{theorem}

We will prove this theorem in two ways: first, by comparing it to the Jantzen filtration of a Verma module using a functor introduced by Backelin, and second, by comparing it to the weight filtration on the corresponding mixed twistor $\mc{D}$-module. 

\subsection{Whittaker vectors in completed deformed Verma modules}
\label{sec: deformed Backelin}

In this section we lift results of Kostant \cite{Kostant} on Whittaker vectors in completed Verma modules to the deformed case.

We return to the setting where $A$ is a local $\C$-algebra with maximal ideal $\mf m$ and fix a map $\varphi:U(\h)\to A$ with $\varphi(\mf{h}) \subseteq \mf{m}$.
Recall that for an $A$-module $M$ we define an ascending filtration $M_{-i}=M\mf m^i$ with associated graded $\gr_iM=M_{i}/M_{i+1}.$ Let $\lambda\in \h^*.$ The deformed Verma module $M_A(\lambda)$ has deformed weight spaces of the form 
\[
M_A(\lambda)^{\lambda-\rho+\mu}\cong U(\bar{\mf n})_{\mu}\otimes A
\]
for a weight $\mu$ of $U(\bar{\mf{n}})$, and for each $i\leq 0$ there is a $\mf g$-equivariant isomorphism $M(\lambda)\cong \gr_i M_A(\lambda).$
We consider the $(\n,A)$-bimodule $U(\n)^*\otimes A$ where $\n$ acts on $U(\n)^*$ by the contragredient action and $A$ by right multiplication. Then we get the following deformed version  of \cite[Prop. 3.8]{Kostant}.


\begin{proposition}\label{prop:descriptionofsimpledeformedvermaasnmodule}
Let $\lambda\in \mf h^*.$ If $M(\lambda)$ is simple, then there is an isomorphism of $(\n,A)$-bimodules
$$\beta: M_A(\lambda)\to U(\n)^*\otimes A.$$ 
\end{proposition}
\begin{proof}
Denote by $v=1\otimes 1 \in M_A(\lambda)$ the highest weight generator and by $l:M_A(\lambda)\to A$ the $A$-linear map determined by $l(v)=1$ and $l(n)=0$ for all $n\in M_A(\lambda)^{\mu}$ for $\mu\neq \lambda-\rho.$

Let $m\in M_A(\lambda).$ Define a $\C$-linear map $\beta_m\in \Hom_\C(U(\mf n), A)$ by $\beta_m(u)=l(\widecheck{u}m)$ for $u\in U(\n).$ Here $\widecheck{uw}=\widecheck{w}\widecheck{u}$ for $u,w\in U(\n)$ and $\widecheck{x}=-x$ for $x\in \mf n.$ 

For $x\in \n,$ $u\in U(\n)$ and $m\in M_A(\lambda)$, we have $$\beta_{xm}(u)=l(\widecheck{u}xm)=-l(\widecheck{u}\widecheck{x}m) = -l(\widecheck{xu}m)=\beta_m(-xu) = x\cdot \beta_{m}(u)$$ where $x\cdot \beta_{m}$ denotes the contragredient action of $x$ on $\beta_{m}.$ Moreover, let $m\in M_A(\lambda)^{\lambda-\rho -\mu}$ and $u\in U(\n)_{\mu'}$ for $\mu'\neq \mu.$ Then $\beta_m(u)=l(\widehat{u}m)=0$ since $\widehat{u}m\in M_A(\lambda)^{\neq \lambda - \rho}.$ It follows that $\beta: m\mapsto \beta_m$ is a map of $(\n, A)$-bimodules with image in $U(\n)^*\otimes A=\Hom_\C(U(\mf n), \C)\otimes A\subset \Hom_\C(U(\mf n), A)$ which maps $M_A(\lambda)^{\lambda-\rho-\mu}$ to $U(\n)^*_{\mu}\otimes A.$

By reducing mod $\mf m$ \eqref{eq: reduction map}, we obtain a commutative diagram 
\[\begin{tikzcd}
	{M_A(\lambda)^{\lambda-\rho-\mu}} & {U(\mf n)^*_{\mu}\otimes A} \\
	{M(\lambda)_{\lambda-\rho-\mu}} & {U(\mf n)^*_{\mu}}
	\arrow["\beta", from=1-1, to=1-2]
	\arrow["\mathrm{red}", from=1-1, to=2-1]
	\arrow["\mathrm{red}", from=1-2, to=2-2]
	\arrow["\overline\beta", from=2-1, to=2-2]
\end{tikzcd}\]
where $\overline\beta$ is an isomorphism by \cite[Prop. 3.8]{Kostant} using that $M(\lambda)$ is simple. Since $A$ is local and the weight spaces are finitely generated free $A$-modules, Nakayama's lemma implies that $\beta$ is an isomorphism. 
\end{proof}
For $M\in \bimod{\h}{A}$ we define the \emph{completion} $\overline{M}$ by
\begin{equation}
    \label{eq: completion definition}
    \overline{M}=\prod_{{\mu}\in \h^*}M^{{\mu}}.
\end{equation}
Let $\eta\in \mf n^*$ and
\begin{equation}
    \label{eq: U eta}
    U_\eta(\n):=\ker(\eta: U(\n)\to \C).
\end{equation} 
Moreover, we define
\begin{align}
\label{eq: Whittaker vectors}
    \Wh_\eta(M)&=\{m\in M\,\mid\, (X-\eta(X))m=0 \text{ for all }X\in \mf n\},\\
\label{eq: generalised Whittaker vectors}
    \Gamma_\eta(M)&=\{m\in M\,\mid\, \exists k\geq 0  \text{ with } U_\eta(\n)^km=0\}, \text{ and}\\
\label{eq: Xi}
    \Omega_\eta(M)&= U(\mf g)\Wh_\eta(M).
\end{align}
We have $\Omega_\eta(M)\subset \Gamma_\eta(M)$ by \cite[Prop. 4.2.1]{Kostant}. $\Gamma_\eta$ admits a filtration by
 \begin{equation}
     \label{eq: Whittaker level}
     \Gamma_\eta(M)_n=\{m\in M\,\mid\, U_\eta(\n)^n m=0\} =\Hom_{U(\n)}(U(\n)/U_\eta(\n)^n,M).
 \end{equation}
Moreover, $\Omega_\eta$ admits a filtration by $\Omega_\eta(M)_n=U(\g)_{\leq n}\Wh_\eta(M)$ where $U(\g)_{\leq n}\subset U(\g)$ is the subspace of all terms of order up to $n\geq 0.$ We abbreviate $\overline{\Wh}_\eta(M)=\Wh_\eta(\overline{M}), \overline{\Gamma}_\eta(M)=\Gamma_\eta(\overline{M})$ and so on.
\begin{theorem}
\label{thm: standards to standards nondegenerate case}
Let $\lambda\in \mf h^*$ and $\eta\in \mf n^*$ be non-degenerate. Assume that $M(\lambda)$ is simple. 
\begin{enumerate}
    \item There is an isomorphism of $(Z(\g)\otimes U(\mf n), A)$-bimodules $$ A_{\lambda - \rho} \cong\overline{\Wh}_\eta(M_A(\lambda)).$$
Where $Z(\g) \otimes U(\n)$ acts on the left hand side via
\[
z \otimes u \cdot a = (\chi_\lambda + \varphi \circ p)(z) \eta(u) a,
\]
where $\chi_\lambda$ and $p$ are as in \eqref{eq: inf char for g} and \eqref{eq: HC hom}. (See also Definition \ref{def: deformed standard Whittaker}).
    \item Assume that $A$ is a PID. Then there is an isomorphism of $(\g,A)$-bimodules $$M_A(\lambda,\eta)=U(\g)\otimes_{Z(\g)\otimes U(\n)}A_{\lambda - \rho} \cong\overline \Omega_\eta( M_A(\lambda))=\overline{\Gamma}_\eta(M_A(\lambda)).$$
\end{enumerate}
\end{theorem}
\begin{proof}
(1) By Proposition \ref{prop:descriptionofsimpledeformedvermaasnmodule} there is an isomorphism of $(\n,A)$-bimodules $$\overline{M_A(\lambda)}\cong \overline{U(\mf n)^*\otimes A}=\Hom_\C(U(\mf n),A).$$ Now $U(\n)=\C \cdot 1 \oplus U_{\eta}(\mf n)$ and $U(\mf n)/U_{\eta}(\mf n)\cong \C.$ We obtain
$$\overline{\Wh}_\eta(M_A(\lambda))=\Wh_\eta(\Hom_\C(U(\n),A))=\Hom_\C(U(\n)/U_{\eta}(\n),A)\cong A_\eta$$
as a $(\n,A)$-bimodule. (Here $A_\eta$ is the free rank $1$ $A$-module with $\n$-action given by $\eta$.) Now $Z(\mf g)$ acts on $M_A(\lambda)$ and $\overline{\Wh}_\eta(M_A(\lambda))$ in the same way. The first statement follows.

(2) Abbreviate $W=\overline{\Wh}_\eta(M_A(\lambda)),$ $M=U(\g)\otimes_{Z(\g)\otimes U(\n)}W$, and $N=\overline\Gamma_\eta(M_A(\lambda)).$
Then $M=\Gamma_\eta(M)=\Omega_\eta(M)$ and $N=\Gamma_\eta(N)\supseteq \Omega_\eta(N).$

Multiplication induces a map of $(\g,A)$-bimodules $$\phi:M\to \Omega_\eta(N)\hookrightarrow N.$$ 
The map $\phi$ becomes an isomorphism mod $\mf m$ by \cite[Thm. 3.8, Lem. 3.9, Thm. 4.4]{Kostant}. As an $A$-module, $M$ is torsion free since  $U(\mf g)$ is free as an $Z(\mf g)\otimes U(\mf n)$-module \cite[Lem. 4.1]{BRpairings} and $N$ is a torsion free $A$-module because $M_A(\lambda)$ is torsion free. Now $\phi$ induces an isomorphism $\Omega_\eta(M)_n\cong  \Omega_\eta(N)_n$ by Nakayama's lemma: The statement is true when reducing modulo $\mf m,$ and both modules are finitely generated free $A$-modules. Here we use that a finitely generated torsion free module over the local PID $A$ is free.
Hence $\phi$ is injective with image $\Omega_\eta(N).$ The isomorphism $\overline{M_A(\lambda)}\cong \Hom_\C(U(\n),A)$ shows that 
\begin{align*}
    \Gamma_\eta(N)_n&\cong\Hom_{U(\mf n)}(U(\mf n)/U_\eta(\mf n)^n,\Hom_\C(U(\mf n), A))\\
    &\cong \Hom(U(\mf n)/U_\eta(\mf n)^n, A)
\end{align*}
is finitely generated as an $A$-module. It follows that $\Gamma_\eta(M)_n$ is also finitely generated. We can hence apply Nakayama's lemma again to see that $\phi$ induces an isomorphism $\Gamma_\eta(M)_n\cong \Gamma_\eta(N)_n.$ It follows that $\phi$ is an isomorphism.
\end{proof}

\begin{remark}
The statements probably hold true without assuming that $M(\lambda)$ is simple or that $A$ is a PID. However, the condition that $A$ is local is essential.
\end{remark}

\begin{theorem}
\label{thm: deformed Vermas go to standard Whittakers}
Let $\lambda\in \mf h^*$ such that the $\mf{l}_\eta$-Verma module of highest weight $\lambda - \rho_\eta$, $M^{\levi_\eta}(\lambda)$, is simple. Assume that $A$ is a PID. 
Then there is an isomorphism
$$M_A(\lambda,\eta)\cong \overline\Gamma_\eta(M_A(\lambda)).$$
\end{theorem}
\begin{proof}
Our proof uses techniques similar to the proof of \cite[Lem. 6.5]{Backelin}. 

We first fix some notation. For $i=1,2$ denote by $\mf l_i\subset \mf p_i\subset \mf g$ the Levi and standard parabolic corresponding to the simple roots $\alpha$ with $\eta(X_\alpha)\neq 0$ and  $\eta(X_\alpha)= 0,$ respectively. In particular, $\levi_1=\levi_\eta$ in the notation of \eqref{eq: all the eta subalgebras}. Let $\rho_i$ be the half-sum of positive roots in the corresponding root systems.
Denote the opposite parabolic by $\overline{\mf p}_i.$ Let $\mf n_i=\mf n \cap \mf l_i$ and $\overline{\mf n}_i=\overline{\n} \cap \mf l_i.$ Denote by $\overline{\Gamma}_\eta^i$ the functor \eqref{eq: generalised Whittaker vectors} for $\eta|_{\mf{n}_i}$.
Then the nil radical of $\overline{\mf p}_i$ is $\overline{\mf u}_i=\overline{\mf n_i} +[\overline{\mf n_1},\overline{\mf n_2}]$ and $\mf l_i$ acts on it with the adjoint action. Moreover, denote by $\mf z_i=\mf z(\mf l_i)$ the respective centers. It follows that $\mf h=\mf z_1 +\mf z_2.$ 

Abbreviate $M=M_A(\lambda).$ For $\lambda_i\in \mf z_i^*$ denote by $M^{\lambda_i}$ the corresponding deformed weight spaces. Hence, we obtain the deformed $\mf h$-weight spaces of $M$ as $M^{\lambda_1,\lambda_2}=M^{\lambda_1}\cap M^{\lambda_2}.$ 

The deformed weight spaces with respect to $\mf z_i$ admit the following explicit description. We have
\begin{align*}
    M&=U(\mf g)\otimes_{U(\mf b)}A_{\lambda - \rho}\\
    &\cong U(\mf g)\otimes_{U(\mf p_i)}U(\mf l_i)\otimes_{U(\mf l_i\cap \mf b)}A_{\lambda - \rho}\\
    &= U(\mf g)\otimes_{U(\mf p_i)} M_A^{\mf l_i}(\lambda)\\
    &= U(\overline{\mf u}_i)\otimes_\C M_A^{\mf l_i}(\lambda)
\end{align*}
where the last isomorphism holds on the level of $(\mf l_i,A)$-bimodules and $M_A^{\mf{l}_i}(\lambda)$ refers to the deformed Verma module for $\mf{l}_i$ of highest weight $\lambda - \rho_i$. We hence obtain for $\mu_i \in \mf{z}_i^*$ the explicit description
$M^{\mu_i}=U(\overline{\mf u}_i)_{\mu_i-(\lambda-\rho)|_{\mf z_i}}\otimes M_A^{\mf l_i}(\lambda).$ (This follows from the description of $\mf{z}_1$ in \eqref{eq: center of levi} and the analogous statement for $\mf{z}_2$.)

In particular, for the $\mf{l}_1$-module $M^{\mu_1}$, we have 
\begin{align*}
    \overline\Gamma^1_\eta(M^{\mu_1})
    &=\overline\Gamma^1_\eta(U(\overline{\mf u}_1)_{\mu_1-(\lambda-\rho)|_{\mf z_1}}\otimes M_A^{\mf l_1}(\lambda))\\
    &=U(\overline{\mf u}_1)_{\mu_1-(\lambda-\rho)|_{\mf z_1}}\otimes \overline\Gamma^1_\eta(M^{\mf l_1}_A(\lambda))\\
    &=U(\overline{\mf u}_1)_{\mu_1-(\lambda-\rho)|_{\mf z_1}}\otimes M_A^{\mf{l}_1}(\lambda, \eta)
\end{align*}
where the second equality follows since  $U(\overline{\mf u}_1)_{\mu_1-(\lambda-\rho)|_{\mf z_1}}$ is finite-dimensional and $\Gamma^1_\eta$ commutes with tensor products with finite-dimensional modules by \cite[Prop. 3]{AriasBackelin} (there the statement is proven in the undeformed case, but translates unchanged), and the third equality follows from Theorem \ref{thm: standards to standards nondegenerate case}, using that $\eta$ is non-degenerate for $\mf{l}_1$ and $M^{\mf{l}_1}(\lambda)$ is assumed to be simple. 

Moreover, for a $\mf{z}_2$-weight $\mu_2 \in \mf{z}_2^*$, the deformed weight space $M^{\mu_2}$ is $\mf{l}_2$-stable, and we have
\begin{align*}
    \overline{\Gamma}^2_\eta(M^{\lambda_2})
    &=U(\overline{\mf u}_2)_{\mu_2-(\lambda-\rho)|_{\mf z_2}}\otimes \overline \Gamma^2_\eta(M^{\mf l_2}(\lambda))\\
    &=U(\overline{\mf u}_2)_{\mu_2-(\lambda-\rho)|_{\mf z_2}} \otimes \overline{\Gamma}^2_0(M_A^{\mf{l}_2}(\lambda)) \\
    &= M^{\mu_2}.
\end{align*}
Here the second equality follows from the fact that $\eta$ vanishes on $\mf{n}_2$ and the third equality holds because there can only be finitely many deformed weight spaces containing highest weight vectors, so $\mf{n}_2$-finite vectors can only be supported in finitely many weight spaces. See \cite[Rmk. 3.3]{Backelin}.

Hence, in total we get an inclusion
\begin{align*}
    \overline\Gamma_\eta(M)&=\Gamma_\eta(\overline M)=
    \Gamma_\eta^{1}\Gamma_0^{2}(\prod_{\mu_2}\prod_{\mu_1}M^{\mu_1,\mu_2})\\
    &\subseteq \Gamma_\eta^{1}(\prod_{\mu_2}\overline\Gamma_0^{2}(M^{\mu_2}))=\Gamma_\eta^{1}(\prod_{\mu_2}M^{\mu_2})\\
    &=\Gamma_\eta^{1}(\prod_{\mu_2}\bigoplus_{\mu_1}M^{\mu_1,\mu_2})=\bigoplus_{\mu_1}\overline\Gamma_\eta^{1}(M^{\mu_1})\\ &=\bigoplus_{\mu_2}U(\overline{\mf u}_1)_{\lambda_1-(\lambda-\rho)|_{\mf z_1}}\otimes_\C M^{\mf{l}_1}_A(\lambda, \eta).
\end{align*}
In particular, by considering the summand $U(\overline{\mf u}_1)_{0}\otimes_\C M^{\mf{l}_1}_A(\lambda, \eta)=M_A^{\mf{l}_1}(\lambda, \eta),$ we obtain a natural multiplication map
$$\phi: M_A(\lambda,\eta)=U(\mf g)\otimes_{U(\mf p)} M^{\mf{l}_1}_A(\lambda, \eta)\to \overline\Gamma_\eta(M)$$
By writing $M_A(\lambda,\eta)=U(\overline{\mf u}_1)\otimes M^{\mf{l}_1}_A(\lambda, \eta)$ we see that $M_A(\lambda,\eta)$ has the same decomposition as the space containing $\overline\Gamma_\eta(M)$ considered above. It follows that $\phi$ is an isomorphism.
\end{proof}
\begin{remark} Actually, the same proof shows that $$\overline\Gamma_\eta(U(\mf g)\otimes_{U(\mf p_1)} M)=U(\mf g)\otimes_{U(\mf p_1)}\overline\Gamma^{1}_\eta(M)$$ for every module $M$ in deformed category $\cO_A$ of $\mf l_1.$ For this one uses that for each module $N$ in category $\cO_A(\mf g)$ we have $N^{\mu_2}$ is in category $\cO_A(\mf l_2)$ for each $\lambda_2\in \mf z_2^*,$ see \cite[Thm. 2.3.2]{eberhardtGradedGeometricParabolic2018}. 
\end{remark}

\subsection{The Jantzen conjectures via Category $\mc{O}$} 
\label{sec: Jantzen conjecture via category O}

Using the results of Section \ref{sec: deformed Backelin}, we can prove the Jantzen conjectures for Whittaker modules (Theorem \ref{thm: Jantzen conjecture Whittaker}) using the Jantzen conjectures for Verma modules (Theorem \ref{thm: Jantzen conjecture}). We dedicate this section to doing so.

We have established that for a local PID $A$ with distinguished homomorphism $\varphi: U(\mf{h}) \rightarrow A$ sending $\mf{h}$ into the maximal ideal $\mf{m}$, the functor 
\[
\overline{\Gamma}_\eta: \mc{O}_A \rightarrow \mf{g} \mathrm{-mod-}A
\]
of \eqref{eq: generalised Whittaker vectors} sends $M_A(\lambda)$ to $M_A(\lambda, \eta)$, where $\lambda \in \mf{h}^*$ is chosen such that the $\mf{l}_\eta$-Verma module $M^{\mf{l}_\eta}(\lambda)$ is simple (Theorem \ref{thm: deformed Vermas go to standard Whittakers}). For every $M_A(\lambda, \eta)$, such a $\lambda$ exists by \eqref{eq: standard Whittakers in an orbit are isomorphic}. 

Moreover, $\overline{\Gamma}_\eta$ also sends $M_A^\vee(\lambda)$ to $M_A^\vee(\lambda, \eta)$. This follows from the fact that the completion \eqref{eq: completion definition} of a deformed dual Verma module is canonically isomorphic to the $A$-linear dual of a deformed Verma module; i.e.
\[
\overline{M_A^\vee(\lambda)} = \Hom_A(M_A(\lambda), A).
\]
Hence the functor $\overline{\Gamma}_\eta$ sends the canonical $(\mf{g},A)$-module homomorphism 
\[
\psi: M_A(\lambda) \rightarrow M_A^\vee(\lambda)
\]
of \eqref{eq: canonical map of vermas} to the canonical $(\mf{g},A)$-module homomorphism 
\[
\psi_\eta: M_A(\lambda, \eta) \rightarrow M_A^\vee(\lambda, \eta)
\]
of \eqref{eq: canonical map deformed Whittaker}. 

By construction, the functor $\overline{\Gamma}_\eta$ has the property that for any $a \in A$ and $(\mf{g}, A)$-bimodule $M$, 
\[
\overline{\Gamma}_\eta(Ma) = \overline{\Gamma}_\eta(M)a,
\]
so $\overline{\Gamma}_\eta$ maps the Jantzen filtration of $M_A(\lambda)$ (Definition \ref{def: Jantzen filtration of Verma}) to the Jantzen filtration of $M_A(\lambda, \eta)$ (Definition \ref{def: Jantzen filtration of a standard Whittaker}). Specializing to the local $\C$-algebra in \eqref{eq: completion of T}, this proves Theorem \ref{thm: Jantzen conjecture Whittaker}.(ii). 

In \cite[Cor. 2]{AriasBackelin}, it is shown that the functor $\overline{\Gamma}_\eta$ is a quotient functor onto its essential image. In particular, this implies that all morphisms in $\mc{N}$ between standard Whittaker modules are obtained from morphisms between Verma modules. Moreover, because $\overline{\Gamma}_\eta$ is exact, it sends injective morphisms to injective morphisms. Hence all embeddings of standard Whittaker modules are of the form 
\[
M(\mu, \eta) \xhookrightarrow{\overline{\Gamma}_\eta(\iota)} M(\lambda, \eta),   
\]
where $\iota: M(\mu) \hookrightarrow M(\lambda)$ is an embedding in category $\mc{O}$. Hence Theorem \ref{thm: Jantzen conjecture Whittaker}.(i) follows from Theorem \ref{thm: Jantzen conjecture}(i). 

\section{$\mc{D}$-Modules}
\label{sec: D-modules}
\subsection{Preliminaries} \label{sec: preliminaries} We start by listing our geometric conventions and recording the necessary background results. 
\subsubsection{$\mc{D}$-module notation and conventions}
\label{sec: D-module motation and conventions}
We adopt the following notational conventions for $\mc{D}$-modules. Our conventions mostly\footnote{In places where they differ, we explicitly state definitions. If definitions are not explicitly stated, they are the same as the sources above.} follow \cite{D-modulesnotes, Localization}. 
\begin{itemize}
    \item If $Y$ is a smooth algebraic variety, $\mc{O}_Y$ is the sheaf of regular functions on $Y$ and $\mc{D}_Y$ is the sheaf of differential operators on $Y$. 
    \item We denote by $\mc{M}_{qc}(\mc{D}_Y), \mc{M}_{coh}(\mc{D}_Y), \mc{M}_{hol}(\mc{D}_Y)$ the categories of quasi-coherent, coherent, and holonomic $\mc{D}_Y$-modules on Y, respectively. We denote by $D^b_{qc}(\mc{D}_Y), D^b_{coh}(\mc{D}_Y), D^b_{hol}(\mc{D}_Y)$ the corresponding bounded derived categories. 
    \item For a morphism $f:Y \rightarrow Z$ of smooth algebraic varieties, we have functors 
    \begin{align*}
        f^+: \mc{M}_{qc}(\mc{D}_Z) &\rightarrow \mc{M}_{qc}(\mc{D}_Y) \\
        f^!: D^b_{qc}(\mc{D}_Z) &\rightarrow D^b_{qc}(\mc{D}_Y) \\
        f_+: D^b_{qc}(\mc{D}_Y) &\rightarrow D^b_{qc}(\mc{D}_Z) \\
        f_!: D^b_{coh}(\mc{D}_Y) &\rightarrow D^b_{coh}(\mc{D}_Z) 
    \end{align*}
    The functors $f^+, f^!, f_+$ are defined in the standard way (see \cite{D-modulesnotes}), and the $!$-pushforward is the twist of the $+$-pushforward by duality:
    \[
    f_! := \mathbb{D}_Z \circ f_+ \circ \mathbb{D}_Y,
    \]
    for the duality functor $\mathbb{D}_\Box: D^b_{coh}(\mc{D}_\Box) \rightarrow D^b_{coh}(\mc{D}_\Box)$, $\Box = Y, Z$, given by 
    \[
    \mathbb{D}_\Box (\mc{V}^\bullet) := RHom_{\mc{D}_\Box}(\mc{V}^\bullet, \mc{D}_\Box) [\dim \Box].
    \]
    \item If $i:Y \rightarrow Z$ is an immersion, then $i_+$ is the right derived functor of the left exact functor 
    \[
    H^0 \circ i_+ \circ [0]: \mc{M}_{qc}(\mc{D}_Y) \rightarrow \mc{M}_{qc}(\mc{D}_Z),
    \]
    where $[0]: \mc{M}_{qc}(\mc{D}_Y) \rightarrow D^b_{qc}(\mc{D}_Y)$ sends a $\mc{D}_Y$-module $\mc{V}$ to the complex $0 \rightarrow \cdots \rightarrow 0 \rightarrow \mc{V} \rightarrow 0 \rightarrow \cdots 0$ concentrated in degree $0$. In this setting, we refer to $H^0 \circ i_+ \circ [0]$ as $i_+$. 
    \item If $\mc{V} \in \mc{M}_{hol}(\mc{D}_Y)$, then $\mathbb{D}_Y(\mc{V}[0]) \in D^b_{hol}(\mc{D}_Y)$ has holonomic cohomology, and $H^p(\mathbb{D}_Y(\mc{V}[0]))=0$ for all $p \neq 0$. In this case, $\mathbb{D}_Y$ descends to a functor on holonomic $\mc{D}_Y$-modules, which we refer to by the same symbol: 
    \begin{align*}
    \mathbb{D}_Y: \mc{M}_{hol}(\mc{D}_Y) &\rightarrow \mc{M}_{hol}(\mc{D}_Y)\\
    \mc{V} &\mapsto H^0( RHom_{\mc{D}_Y}(\mc{V}[0], \mc{D}_Y) [\dim Y])
    \end{align*}
    \item Hence if $i:Y \rightarrow Z$ is an immersion, we can consider $i_!$ as a functor on holonomic modules:
    \[
    i_!: \mc{M}_{hol}(\mc{D}_Y) \rightarrow \mc{M}_{hol}(\mc{D}_Z).
    \]
    \item If $f:Y \rightarrow Z$ is proper, then $f_+$ preserves the derived category of coherent $\mc{D}$-modules, and on this category $f_! = f_+$. In particular, if $i:Y \rightarrow Z$ is a closed immersion, then the functors $i_!$ and $i_+$ agree on holonomic $\mc{D}_Y$-modules:
\[
i_+=i_!:\mc{M}_{hol}(\mc{D}_Y) \rightarrow \mc{M}_{hol}(\mc{D}_Z).
\]
    \item If $Y$ admits an action by an algebraic group $K$ with finitely many orbits, then we denote by $\mc{M}_{qc}(\mc{D}_Y, K)$ the category of strongly $K$-equivariant quasicoherent $\mc{D}_Y$-modules, and by $\mc{M}_{qc}(\mc{D}_Y, K)_\mathrm{weak}$ the category of weakly $K$-equivariant quasicoherent $\mc{D}_Y$-modules. (See \cite{MP} for a reference on strong and weak equivariance.) 
\end{itemize}

\subsubsection{The flag variety and base affine space}
\label{sec: the flag variety and base affine space}

Let $\mf{g}$ be a complex semisimple Lie algebra, and denote by $G$ the simply connected semisimple Lie group corresponding to $\mf{g}$. Fix a Borel subgroup $B \subset G$, and denote by $N \subset B$ the unipotent radical of $B$. The abstract torus 
\begin{equation}
    \label{eq: H}
    H:= B/N
\end{equation}
does not depend on choice of $B$ \cite[Lem. 6.1.1]{chrissginzburg}. Set 
\begin{equation}
    \label{eq: flag variety and base affine space}
    X:= G/B \text{ and } \widetilde{X}:=G/N.
\end{equation}
The variety $X$ is the flag variety of $\mf{g}$, and we refer to the variety $\widetilde{X}$ as {\bf base affine space}. Both $X$ and $\widetilde{X}$ have left actions of $G$ coming from left multiplication. There is also a natural right action of $H$ on $\widetilde{X}$ by right multiplication, and with respect to this action, the quotient map
\begin{equation}
    \label{eq: projection onto flag variety}
    \pi: \widetilde{X} \rightarrow X
\end{equation}
is a principal $G$-equivariant $H$-bundle. 
\subsubsection{Differential operators on base affine space and the universal enveloping algebra}
\label{sec: differential operators on base affine space and the universal enveloping algebra}
Global differential operators on $\widetilde{X}$ are related to the extended universal enveloping algebra. More specifically, by differentiating the left $G$-action and right $H$-action described above (which clearly commute), one obtains an algebra homomorphism 
\[
U(\mf{g}) \otimes_\C U(\mf{h}) \rightarrow \Gamma(\widetilde{X}, \mc{D}_{\widetilde{X}}).
\]
This homomorphism factors through the quotient \cite[\S2]{BBJantzen}
\begin{equation}
    \label{eq: extended universal enveloping algebra}
    \widetilde{U}:= U(\mf{g}) \otimes_{Z(\mf{g})} U(\mf{h}).
\end{equation}
Here $Z(\mf{g})$ acts on $U(\mf{h})$ via the composition of the Harish-Chandra homomorphism \eqref{eq: HC hom} with the $\rho$-shift endomorphism
\begin{equation}
    \tau: U(\h) \rightarrow U(\h)
\end{equation}
which sends $h \in \h$ to $\tau(h)=h - \rho(h)$. We refer to the algebra $\widetilde{U}$ as the {\bf extended universal enveloping algebra}. By construction, $\widetilde{U}$ contains both $U(\g)$ and $U(\h)$ as subalgebras, and the subalgebra $U(\h) \subset \widetilde{U}$ is the center. Moreover, there is a natural action of $W$ on $\widetilde{U}$ coming from the $W$-action\footnote{Recall that it is our convention to work with the usual action of $W$ on $\h^*$, not the dot action.} on $U(\h)$, and the Harish-Chandra isomorphism\footnote{Here $p$ is the projection in \eqref{eq: HC hom}.} $\tau \circ p: Z(\mf{g}) \xrightarrow{\simeq} U(\h)^W$ guarantees that $\widetilde{U}^W \simeq U(\mf{g})$.   

Global sections of $\mc{D}_{\widetilde{X}}$-modules have the structure of $\widetilde{U}$-modules via the map 
\begin{equation}
    \label{eq: alpha}
    \alpha: \widetilde{U} \rightarrow \Gamma(\widetilde{X}, \mc{D}_{\widetilde{X}}).
\end{equation}

\subsubsection{Monodromic $\mc{D}$-Modules}
\label{sec: monodromic D-modules}
The $\mc{D}$-modules which arise in the geometric construction of the Jantzen filtration have an additional structure - they are $H$-monodromic. Using $H$-monodromic $\mc{D}$-modules on base affine space, one can study $\mf{g}$-modules of generalized infinitesimal character. (In contrast, using the more classical approach of modules over sheaves of twisted differential operators on the flag variety, one can only study $\mf{g}$-modules of strict infinitesimal character.) 

Recall the principal $H$-bundle $\pi:\widetilde{X} \rightarrow X$ from \eqref{eq: projection onto flag variety}. An {\bf $H$-monodromic $\mc{D}$-module on $X$} is a weakly $H$-equivariant $\mc{D}_{\widetilde{X}}$-module. Such $\mc{D}$-modules have an alternate characterization in terms of $H$-invariant sections of $\pi_\bullet \mc{D}_{\widetilde{X}}$. Specifically, the right $H$-action on $\widetilde{X}$ induces an $H$-action on $\pi_\bullet \mc{D}_{\widetilde{X}}$ by algebra automorphisms, where $\pi_\bullet$ denotes the sheaf-theoretic direct image functor. We define 
\begin{equation}
\label{eq: D tilde}
\widetilde{\mc{D}}:= [\pi_\bullet \mc{D}_{\widetilde{X}}]^H \subset \pi_\bullet \mc{D}_{\widetilde{X}}
\end{equation}
to be the sheaf of $H$-invariant sections of $\pi_\bullet\mc{D}_{\widetilde{X}}$. This is a sheaf of algebras. There is an equivalence of categories\footnote{Here $\mc{M}_{qc}(\widetilde{\mc{D}})$ denotes the category of $\widetilde{\mc{D}}$-modules which are quasicoherent as $\mc{O}_X$-modules.} \cite[\S1.8.9, 2.5]{BBJantzen}
\begin{equation}
    \label{eq: two notions of monodromicity}
    \mc{M}_{qc}(\mc{D}_{\widetilde{X}}, H)_{\mathrm{weak}} \simeq \mc{M}_{qc}(\widetilde{\mc{D}}).
\end{equation}

Moreover, because the left $G$-action and right $H$-action commute, the image of the homomorphism $\alpha$ in \eqref{eq: alpha} is contained in $\Gamma(\widetilde{X}, \mc{D}_{\widetilde{X}})^H = \Gamma(X, \widetilde{\mc{D}})$. The resulting map is an algebra isomorphism \cite{BoBr}
\begin{equation}
    \label{eq: alpha isomorphism version}
    \alpha: \widetilde{U} \xrightarrow{\simeq} \Gamma(X, \widetilde{\mc{D}}). 
\end{equation}
Hence the global sections of $H$-monodromic $\mc{D}$-modules on $X$ have the structure of $\widetilde{U}$-modules. 

\begin{remark} \label{rem: relationship to TDOs}
    (Relationship to twisted differential operators) There is a natural embedding $S(\mf{h})\hookrightarrow \widetilde{\mc{D}}$ coming from the action of the Lie algebra of $H$, whose image lies in (and, in fact, is equal to) the center of $\widetilde{\mc{D}}$ \cite[\S2.5]{BBJantzen}. For $\lambda \in \mf{h}^*$, denote by $\mf{m}_\lambda \subset S(\mf{h})$ the maximal ideal corresponding to $\lambda + \rho$\footnote{Our $\rho$-shift conventions are chosen so that the definition of $\mc{D}_\lambda$ aligns with that in \cite{Localization, TwistedSheaves, Romanov, BrownRomanov}. These conventions differ from those in \cite{BBJantzen} by a negative $\rho$-shift.}. The sheaf $\mc{D}_\lambda:= \widetilde{D}/\mf{m}_\lambda \widetilde{\mc{D}}$ is a sheaf of twisted differential operators (TDOs) on the flag varity $X$. Therefore, $\widetilde{\mc{D}}$-modules on which $\mf{m}_\lambda$ acts trivially can be identified with $\mc{D}_\lambda$-modules. 
\end{remark}

The relationships described above can be summarized in the following commutative diagram. 
\begin{equation}
\label{eq: big diagram}
    \begin{tikzcd}
        \mc{M}_{qc}(\mc{D}_{\widetilde{X}},H)_{\mathrm{weak}} \arrow[rr, "\text{forget equiv.}"] \arrow[d, "\Gamma"'] & & \mc{M}_{qc}(\mc{D}_{\widetilde{X}})  \arrow[rr, "\pi_\bullet"] & & \mc{M}_{qc}(\pi_\bullet \mc{D}_{\widetilde{X}}) \arrow[d, "\text{restrict}"] \\
        \mc{M}(\widetilde{U}) & & & & \mc{M}_{qc}(\widetilde{\mc{D}}) \arrow[llll, "\Gamma"']
    \end{tikzcd}
\end{equation}
The composition of the top two horizontal arrows and the right-most vertical arrow is the equivalence in \eqref{eq: two notions of monodromicity}.

\subsection{Whittaker $\mc{D}$-modules}
\label{sec: Whittaker D-modules}
In this section, we describe the $\mc{D}$-modules which can be used to study the standard and costandard Whittaker modules of \S \ref{sec: standard/costandard modules and contravariant pairings}. 

Let $\mf{b}=\Lie B$ be the Lie algebra of $B$ and $\mf{n}=\Lie N$ the Lie algebra of $N$. Choose a character $\eta \in \ch \mf{n}$. In \cite{Romanov, TwistedSheaves}, the category $\mc{M}_{coh}(\mc{D}_\lambda, N, \eta)$ of {\em $\eta$-twisted Harish-Chandra sheaves} on $X$ is studied. An object in this category is a coherent $\mc{D}_\lambda$-module\footnote{Here $\lambda \in \mf{h}^*$ is fixed and $\mc{D}_\lambda$ is the TDO in Remark \ref{rem: relationship to TDOs}.}  $\mc{V}$ which is $N$-equivariant as an $\mc{O}_X$-module and satisfies the additional properties that the morphism $\mc{D}_\lambda \otimes \mc{V} \rightarrow \mc{V}$ is $N$-equivariant, and if $\mu:\mf{n} \rightarrow \End{\mc{V}}$ denotes the differential of the $N$-action and $\pi: \mc{D}_\lambda \rightarrow \End{\mc{V}}$ the $\mc{D}_\lambda$-action, then for any $\xi \in \mf{n}$, 
\[
\pi(\xi) = \mu(\xi) + \eta(\xi). 
\]
Twisted Harish-Chandra sheaves are holonomic $\mc{D}_\lambda$-modules, and hence have finite length \cite[Lem. 1.1]{TwistedSheaves}. 

In \cite{TwistedSheaves} simple objects in $\mc{M}_{coh}(\mc{D}_\lambda, N, \eta)$ are classified. They are parametrized by pairs $(Q, \tau)$, where $Q$ is an $N$-orbit in $X$ and $\tau$ is an $\eta$-twisted connection on $Q$ (i.e., an irreducible $(\mc{D}_Q, N, \eta)$-module). In particular, an $N$-orbit (Bruhat cell) $C(w)$ admits an $\eta$-twisted connection exactly when $w=w^C$ is the unique longest coset representative in a coset $C \in W_\eta \backslash W$, where $W_\eta \subset W$ is the subgroup generated by the simple roots with the property that $\eta$ does not vanish on the corresponding root subspace, as in \S \ref{sec: standard/costandard modules and contravariant pairings} \cite[Lem. 4.1, Thm. 4.2]{TwistedSheaves}. On such a Bruhat cell $C(w^C)$, the unique $\eta$-twisted connection is the exponential $\mc{D}_{C(w^C)}$-module $\mc{O}_{C(w^C), \eta}$. As an $N$-equivariant $\mc{O}_{C(w^C)}$-module, $\mc{O}_{C(w^C), \eta}$ is isomorphic to $\mc{O}_{C(w^C)}$, but the $\mc{D}_\lambda$-module structure is twisted by $\eta$. 

\begin{remark}
\label{rem: the exponential D-module}
    (The exponential $\mc{D}$-module) The $\eta$-twisted connection $\mc{O}_{C(w^C), \eta}$ is ``exponential'' in the following sense. The exponential function $f(x)=e^x$ on $\mathbb{A}^1$ is the solution of the ordinary differential equation $\partial f = f$, or $(\partial - 1)f =0$. The corresponding $D_{\mathbb{A}^1}$-module is the exponential $\mc{D}_{\mathbb{A}^1}$-module 
    \[
    \mathrm{exp}:= D_{\mathbb{A}^1}/D_{\mathbb{A}^1}(\partial - 1). 
    \]
    Here $D_{\mathbb{A}^1}$ denotes the global differential operators on $\mathbb{A}^1$; i.e., the Weyl algebra. Corresponding to the Lie algebra character $\eta\in \ch \mf{n}$ is an additive character $\eta:N \rightarrow \mathbb{A}^1$, which we call by the same name. We can use the character $\eta$ to construct exponential $\mc{D}$-modules on certain Bruhat cells. In particular, if for all $x \in C(w)$, $\eta|_{\stab_Nx}=1$, then $\eta$ factors through $N/\stab_Nx \simeq C(w)$: 
    \[
    \begin{tikzcd}
        N \arrow[r] \arrow[rr, bend left, "\eta"] & N/\stab_Nx \arrow[r, "\overline{\eta}"] & \mathbb{A}^1.
    \end{tikzcd}
    \]
    For such a Bruhat cell $C(w)$, we can pull back the exponential $D_{\mathbb{A}^1}$-module $\mathrm{exp}$ to an exponential $\mc{D}$-module $\overline{\eta}^! \mathrm{exp}$ on $C(w)$. Lemma 4.1 in \cite{TwistedSheaves} establishes that $\eta|_{\mathrm{stab}_Nx}=1$ precisely when $w=w^C$ is a longest coset representative for some $C \in W_\eta \backslash W$. The $\mc{D}$-modules constructed in this way are the $\eta$-twisted connections $\mc{O}_{C(w^C),\eta}$ described above. 
\end{remark}

The category $\mc{M}_{coh}(\mc{D}_\lambda, N, \eta)$ is a highest weight category \cite[Thm. 7.2]{BRpairings}. By the discussion above, it has finitely many simple objects, parameterized by $W_\eta \backslash W$. Standard and costandard objects are given by $!$- and $+$-pushforwards (respectively) of the exponential $\mc{D}_\lambda$-modules $\mc{O}_{C(w^C), \eta}$ for $C \in W_\eta \backslash W$. When $\lambda \in \mf{h}^*$ is regular and antidominant\footnote{See Section \S\ref{sec: definitions} for our conventions with these definitions.}, the global sections of standard and costandard ($\mc{D}_\lambda, N, \eta)$-modules are the standard and costandard Whittaker modules defined in \S \ref{sec: standard/costandard modules and contravariant pairings}. Specifically, we have the following relation. 
\begin{theorem}
    \label{thm: matching of global sections undeformed}\cite[Thm. 9]{Romanov}, \cite[Lem. 7.3]{BRpairings} Let $\lambda \in \mf{h}^*$ be antidominant, $\eta \in \ch \mf{n}$, and $w^C$ the longest coset representative in $C \in W_\eta \backslash W$. The standard and costandard $(\mc{D}_\lambda, N, \eta)$-modules $i_{C(w^C)!}\mc{O}_{C(w^C), \eta}$ and $i_{C(w^C)+}\mc{O}_{C(w^C), \eta}$ have global sections 
\begin{equation}
    \label{eq: global sections of standards}
    \Gamma(X, i_{C(w^C)!}\mc{O}_{C(w^C), \eta}) = M(w^C\lambda, \eta)
\end{equation}
\begin{equation}
    \label{eq: global sections of standards}
    \Gamma(X, i_{C(w^C)+}\mc{O}_{C(w^C), \eta}) = M^\vee(w^C\lambda, \eta).
\end{equation}
\end{theorem}

Using the set-up in \S \ref{sec: preliminaries}, we can also find standard and costandard Whittaker modules within the global sections of $\mc{D}_{\widetilde{X}}$-modules. This perspective will be necessary in the construction of the Jantzen filtration. Let $Q = C(w^C)$ be a Bruhat cell corresponding to a longest coset representative $w^C \in C \in W_\eta \backslash W$, and denote by $\widetilde{Q} = \pi^{-1} (Q)$ its preimage in $\widetilde{X}$. 
Because the map $\pi:\widetilde{X} \rightarrow X$ is a $G$-equivariant principal $H$-bundle, it preserves categories of $\eta$-twisted Harish-Chandra sheaves under $\mc{D}$-module functors. Hence the $\mc{D}_{\widetilde{Q}}$-module $\pi^+ \mc{O}_{Q, \eta} = \mc{O}_{\widetilde{Q}, \eta}$ obtained by pulling back the exponential $\mc{D}_Q$-module $\mc{O}_{Q, \eta}$ to $\widetilde{Q}$ is a $(\mc{D}_{\widetilde{Q}}, N, \eta)$-module.

Let $i_{\widetilde{Q}}: \widetilde{Q} \hookrightarrow \widetilde{X}$ be inclusion, and consider the $\mc{D}_{\widetilde{X}}$-modules 
\begin{equation}
    \label{eq: standards and costandards on base affine space}
    i_{\widetilde{Q}!}\mc{O}_{\widetilde{Q},\eta} \text{ and } i_{\widetilde{Q} +} \mc{O}_{\widetilde{Q}, \eta}. 
\end{equation}
These are $H$-monodromic $\mc{D}$-modules on $X$, so we can consider the corresponding $\widetilde{\mc{D}}$-modules under the correspondence \eqref{eq: two notions of monodromicity}. Specifically, these are the $\pi_\bullet \mc{D}_{\widetilde{X}}$-modules $\pi_\bullet i_{\widetilde{Q} !} \mc{O}_{\widetilde{Q}, \eta}$ and $\pi_\bullet i_{\widetilde{Q} +} \mc{O}_{\widetilde{Q}, \eta}$, with the restricted action of $\widetilde{\mc{D}} \subset \pi_\bullet \mc{D}_{\widetilde{X}}$. For $\lambda \in \mf{h}^*$, consider the subsheaves 
\begin{equation}
    \label{eq: std and costd on base affine with lambda}
\left(\pi_\bullet i_{\widetilde{Q}!}\mc{O}_{\widetilde{Q},\eta}\right)_\lambda \text{ and } \left(\pi_\bullet i_{\widetilde{Q} +} \mc{O}_{\widetilde{Q}, \eta}\right)_\lambda 
\end{equation}
on which $\mf{m}_\lambda$ acts trivially. By Remark \ref{rem: relationship to TDOs}, these can be considered as $\mc{D}_\lambda$-modules on $X$. For integral $\lambda$, \eqref{eq: std and costd on base affine with lambda} are isomorphic to the standard and costandard (respectively) $\eta$-twisted Harish-Chandra sheaves $i_{Q!} \mc{O}_{Q, \eta}$ and $i_{Q+} \mc{O}_{Q, \eta}$ \cite[Lem. 2.5.4]{BBJantzen}. 

This allows us to study standard and costandard $\eta$-twisted Harish-Chandra sheaves for all integral $\lambda$ simultaneously by working with the $\mc{D}_{\widetilde{X}}$-modules \eqref{eq: standards and costandards on base affine space}. By taking global sections on which the ideal $\widetilde{U}\mf{m}_\lambda \subset \widetilde{U}$ acts trivially, we recover the standard and costandard Whittaker modules of \S \ref{sec: standard/costandard modules and contravariant pairings}. In particular, if for a $\widetilde{U}$-module $M$, we denote by $\left( M \right)_\lambda$ the sub-$U(\g)$-module on which $\widetilde{U} \mf{m}_\lambda$ acts trivially, by diagram \eqref{eq: big diagram} and Theorem \ref{thm: matching of global sections undeformed}, we have the following relationship. 
\begin{lemma}
    \label{lem: realising standard Whittakers on base affine space}
    Let $\lambda \in \h^*$ be antidominant, $\eta \in \ch \mf{n}$, and $w^C$ the longest coset representative in $C \in W_\eta \backslash W$. Denote by $\widetilde{C}(w^C) = \pi^{-1} (C(w^C))$ the preimage of the Bruhat cell $C(w^C)$ under $\pi:\widetilde{X} \rightarrow X$. We have the following isomorphism of $U(\mf{g})$-modules. 
    \[
    \left( \Gamma (\widetilde{X}, i_{\widetilde{C}(w^C)!} \mc{O}_{\widetilde{C}(w^C), \eta} )\right)_\lambda = M(w^C \lambda, \eta) 
    \]
\end{lemma}

\subsection{A geometric Jantzen filtration via monodromy} \label{sec: a geometric Jantzen filtration via monodromy}
In \cite{BBJantzen}, Beilinson--Bernstein construct ``geometric Jantzen filtrations'' of certain $\mc{D}$-modules using monodromy filtrations. In this section, we describe their construction in detail, and explain how the Whittaker $\mc{D}$-modules introduced in \S \ref{sec: Whittaker D-modules} can be made to fit into this framework.

\subsubsection{Monodromy filtrations in abelian categories} \label{sec: monodromy filtrations in abelian categories} Let $\mc{A}$ be an abelian category, $A$ an object in $\mc{A}$, and $s \in \End_\mc{A}(A)$ a nilpotent endomorphism. The Jacobson-Morosov theorem \cite[Prop. 1.6.1]{Deligne} guarantees the existence of a unique increasing exhaustive filtration $\mu_\bullet$ of $A$ satisfying two conditions: 
\begin{enumerate} [label=(\roman*)]
    \item $s \mu_i \subset \mu_{i-2}$ for all $i$, and
    \item the endomorphism $s^k$ induces an isomorphism $\gr^\mu_kA \simeq \gr^\mu_{-k} A$ for all $k$. 
\end{enumerate}
This is the {\bf monodromy filtration} of $A$. The monodromy filtration naturally induces a filtration $J_{!,\bullet}$ on the sub-object $\ker s$ and a filtration $J_{+,\bullet}$ on the quotient $\coker s$. The filtrations $J_{!,\bullet}$ and $J_{+,\bullet}$ admit a description in terms of powers of $s$. Explicitly, if we denote by
\begin{equation}
    \mc{K}^p:=\begin{cases} \ker s^{p+1} & \text{ for } p \geq 0; \\ 0 & \text{ for } p<0, \end{cases} \text{ and } \mc{I}^p:= \begin{cases} \im s^p & \text{ for } p>0; \\ A & \text{ for } p \leq 0 \end{cases} 
\end{equation}
the kernel and image filtrations of $A$, then 
\begin{equation}
    \label{eq: monodromy filtrations as powers of s}
    J_{!,i} = \ker s \cap \mc{I}^{-i} \text{ and } J_{+,i} = (\mc{K}^i + \im s) / \im s. 
\end{equation}

\subsubsection{The maximal extension functor} 
\label{sec: the maximal extension functor} To utilize the monodromy filtration in the setting of holonomic $\mc{D}$-modules, we need a procedure for producing $\mc{D}$-modules with nilpotent endomorphisms. The maximal extension functor accomplishes this. In this section, we recall the maximal extension for $\mc{D}$-modules, which is a special case of the construction in \cite{Bei87}.

Let $f:Y \rightarrow \A^1$ be a regular function on a smooth algebraic variety $Y$. This gives rise to an open-closed decomposition of $Y$:
\begin{equation}
    \label{eq: open-closed}
    U:= f^{-1}(\mathbb{A}^1 - \{0\}) \xhookrightarrow{j} Y \xhookleftarrow{i} f^{-1}(0). 
\end{equation}
For $n \in \mathbb{N}$, denote by 
\[
I^{(n)}:= (\mc{O}_{\mathbb{A}^1 - \{0\} }\otimes \C[s]/s^n) t^s
\]
the free rank one $\mc{O}_{\mathbb{A}^1 - \{0\}} \otimes \C[s]/s^n$-module generated by the symbol $t^s$. The action 
\begin{equation}
    \label{eq: D-module structure on I^n}
    \partial_t \cdot t^s = st^{-1} t^s
\end{equation}
gives $I^{(n)}$ the structure of a $\mc{D}_{\mathbb{A}^1 - \{0\}}$-module. Any $\mc{D}_U$-module $\mc{M}_U$ (where $U = f^{-1}(\A^1 - \{0\})$ as in \eqref{eq: open-closed}) can be deformed using $I^{(n)}$ in the following way. For $n \in \N$, set 
\begin{equation}
    \label{eq: deformed D-module}
f^s \mc{M}_U^{(n)}:= f^+ I^{(n)} \otimes_{\mc{O}_U} \mc{M}_U. 
\end{equation}
We consider this a ``deformation'' because $f^{s}\mc{M}_U^{(1)} = \mc{M}_U$. Note that multiplication by $s\in \C[s]/s^n$ acts nilpotently on $f^s \mc{M}_U^{(n)}$ by construction.

Assume that $\mc{M}_U$ is holonomic. There is a canonical map between the $!$ and $+$ pushforward of $f^s\mc{M}_U^{(n)}$ to $Y$:
\begin{equation}
    \label{eq: canonical map}
    \can: j_!f^s\mc{M}_U^{(n)} \rightarrow j_+ f^s\mc{M}_U^{(n)}.
\end{equation}
We denote the composition of the canonical map with multiplication by $s$ by  
\begin{equation}
\label{eq: s^1(n)}
    s^1(n) := s \circ \can : j_!f^s\mc{M}_U^{(n)} \rightarrow j_+ f^s\mc{M}_U^{(n)}.
\end{equation}
For large enough $n$, the cokernel of $s^1(n)$ stabilizes \cite[Lem. 2.1]{Bei87}. Define the {\bf maximal extension of $\mc{M}_U$} to be the $\mc{D}_Y$-module 
\begin{equation}
    \label{eq: maximal extension definition}
    \Xi_f \mc{M}_U:= \coker s^1(n)
\end{equation}
for $n$ large enough that $\coker s^1(n)$ has stabilized. By construction, $\Xi_f\mc{M}_U$ comes equipped with a nilpotent endomorphism, and $(\Xi_f\mc{M}_U) |_U = \mc{M}_U$. 

This construction defines an exact functor \cite[Lem. 4.2.1(i)]{BBJantzen}
\begin{equation}
    \label{eq: maximal extension functor}
\Xi_f: \mc{M}_{hol}(\mc{D}_U) \rightarrow \mc{M}_{hol}(\mc{D}_Y).
\end{equation}
Moreover, there are canonical short exact sequences \cite[Lem. 4.2.1(ii)']{BBJantzen}
\begin{align}
    \label{eq: canonical ses 1}
    &0 \rightarrow j_!\mc{M}_U \rightarrow \Xi_f\mc{M}_U \rightarrow \coker(\mathrm{can}) \rightarrow 0 \\ 
    \label{eq: canonical ses 2}
    &0 \rightarrow \coker(\mathrm{can}) \rightarrow \Xi_f \mc{M}_U \rightarrow j_+ \mc{M}_U \rightarrow 0
\end{align}
with $j_! = \ker (s: \Xi_f  \rightarrow \Xi_f)$ and $j_+ = \coker(s: \Xi_f \rightarrow \Xi_f)$.

\subsubsection{Geometric Jantzen filtrations}
\label{sec: geometric Jantzen filtrations}
We continue working in the setting of \S\ref{sec: the maximal extension functor}. For any regular function $f: Y \rightarrow \A^1$ and holonomic $\mc{D}_U$-module $\mc{M}_U$, the $\mc{D}_Y$-module $\Xi_f \mc{M}_U$ has a monodromy filtration given by the nilpotent endomorphism $s$. Moreover, the sub- and quotient $\mc{D}_Y$-modules 
\begin{equation}
    j_!\mc{M}_U = \ker (s: \Xi_f \mc{M}_U \rightarrow \Xi_f\mc{M}_U) \text{ and } j_+\mc{M}_U = \coker (s: \Xi_f \mc{M}_U \rightarrow \Xi_f\mc{M}_U)
\end{equation}
obtain the filtrations $J_{!,\bullet}$ and $J_{+,\bullet}$ of \eqref{eq: monodromy filtrations as powers of s}. We call the filtrations $J_{!,\bullet}$ and $J_{+,\bullet}$ the {\bf geometric Jantzen filtrations} of $j_!\mc{M}_U$ and $j_+ \mc{M}_U$, respectively. 

\subsubsection{Geometric Jantzen filtrations on Whittaker $\mc{D}$-modules}
\label{sec: geometric Janten filtrations on Whittaker D-modules}
Now we specialize to the setting of base affine space $\widetilde{X}$. Let $Q \subset X$ be an $N$-orbit admitting an $\eta$-twisted connection; i.e., a Bruhat cell $C(w^C)$ corresponding to a longest coset representative $w^C \in C \in W_\eta \backslash W$, as in \S \ref{sec: Whittaker D-modules}. Let $\widetilde{Q} = \pi^{-1}(Q)$ be the corresponding union of $N$-orbits in $\widetilde{X}$. Fix a positive\footnote{By a {\em positive} weight, we mean a weight $\lambda \in \h^*$ such that $\alpha^\vee(\lambda)>0$ for all positive roots $\lambda \in \Sigma^+$.} regular integral weight $\varphi \in \mf{h}^*$.  It is explained in \cite[\S3.4, Lem. 3.5.1]{BBJantzen} that corresponding to $\varphi$ is a unique (up to multiplication by a constant) non-zero $N$-invariant function $f_\varphi: \widetilde{Q} \rightarrow \A^1$ satisfying 
\begin{equation}
    \label{eq: condition of regular function}
    f_\varphi(h\widetilde{x}) = (\exp \varphi)(h) f_\varphi(\widetilde{x})
\end{equation}
for all $h \in H$, $\widetilde{x} \in \widetilde{X}$. Moreover, $f_\varphi$ is a regular function on the closure $\overline{\widetilde{Q}}$ of $\widetilde{Q}$ with the property that $f_\varphi^{-1}(0) = \overline{\widetilde{Q}}\backslash \widetilde{Q}$, and $f_\varphi$ extends to a regular function on $\widetilde{X}$ which we call by the same name\footnote{See proof of Lemma 3.5.1 in \cite{BBJantzen} for an explicit construction of this function. Note that this construction works for any $N$-orbit in $X$, not just those admitting $\eta$-twisted connections.}.

Let $U:= f_\varphi^{-1}(\mathbb{A}^1 - \{0\}) \subset \widetilde{X}$. Because $f_\varphi^{-1}(0) = \overline{\widetilde{Q}}\backslash \widetilde{Q}$, $\widetilde{Q}$ is closed in $U$. The regular function $f_\varphi:\widetilde{X} \rightarrow \A^1$ determines a maximal extension functor 
\begin{equation}
    \label{eq: maximal extension big}
    \Xi_\varphi:= \Xi_{f_\varphi}: \mc{M}_{hol}(\mc{D}_U) \rightarrow \mc{M}_{hol}(\mc{D}_{\widetilde{X}})
\end{equation}
by the process laid out in \S\ref{sec: the maximal extension functor}. Denote by $\mc{M}_{hol}^{\overline{\widetilde{Q}}}(\mc{D}_{\widetilde{X}}) \subset \mc{M}_{hol}(\mc{D}_{\widetilde{X}})$ (resp. $\mc{M}_{hol}^{\widetilde{Q}}(\mc{D}_U) \subset \mc{M}_{hol}(\mc{D}_U)$) the category of holonomic $\mc{D}_{\widetilde{X}}$-modules  (resp. $\mc{D}_U$-modules) supported on $\overline{\widetilde{Q}}$ (resp. $\widetilde{Q}$). By Kashiwara's theorem \cite[Thm. 12.6]{D-modulesnotes}, there are equivalences of categories
\begin{equation}
    \label{eq: supported on orbits}
    \mc{M}_{hol}^{\overline{\widetilde{Q}}}(\mc{D}_{\widetilde{X}}) \simeq \mc{M}_{hol}(\mc{D}_{\overline{\widetilde{Q}}}) \text{ and } \mc{M}_{hol}^{\widetilde{Q}}(\mc{D}_U) \simeq \mc{M}_{hol}(\mc{D}_{\widetilde{Q}}). 
\end{equation}
The functor $\Xi_\varphi$ \eqref{eq: maximal extension big} transforms the subcategory $\mc{M}_{hol}^{\widetilde{Q}}(\mc{D}_U)$ into $\mc{M}_{hol}^{\overline{\widetilde{Q}}}(\mc{D}_{\widetilde{X}})$ \cite[Rmk. 4.2.2 (ii)]{BBJantzen}, so using \eqref{eq: supported on orbits}, we have 
\begin{equation}
    \label{eq: maximal extension small}
    \mc{M}_{hol}(\mc{D}_{\widetilde{Q}}) \simeq 
    \mc{M}_{hol}^{\widetilde{Q}}(\mc{D}_U) \xrightarrow{\Xi_\varphi} \mc{M}_{hol}^{\overline{\widetilde{Q}}}(\mc{D}_{\widetilde{X}}). 
\end{equation}
In this way, the $\mc{D}_{\widetilde{X}}$-modules $i_{\widetilde{Q} !}\mc{O}_{\widetilde{Q}, \eta}$ and $i_{\widetilde{Q}+} \mc{O}_{\widetilde{Q}, \eta}$ of \eqref{eq: standards and costandards on base affine space} obtain geometric Jantzen filtrations from the monodromy filtration on $\Xi_\varphi\mc{O}_{\widetilde{Q}, \eta}$. A priori, these filtrations depend on choice of  $\varphi \in \mf{h}^*$. 

\subsection{Relationship between algebraic and geometric Jantzen filtration}
\label{sec: relationship between algebraic and geometric Janten filtrations}
The geometric Jantzen filtration on $i_{\widetilde{Q}!}\mc{O}_{\widetilde{Q}, \eta}$ \eqref{eq: standards and costandards on base affine space} described in \S \ref{sec: geometric Janten filtrations on Whittaker D-modules} aligns with the algebraic Jantzen filtration on the corresponding standard Whittaker modules (Definition \ref{def: standard Whittaker module}) under the global sections functor. This relationship is not immediately obvious from the definitions, so we dedicate this section to explaining the connection. 

We return to the setting of \S\ref{sec: the maximal extension functor}. Let $f: Y \rightarrow \A^1$ be a regular function on a smooth algebraic variety and $U=f^{-1}(\mathbb{A}^1-\{0\})$. As in \eqref{eq: s^1(n)}, for any holonomic $\mc{D}_U$-module $\mc{M}_U$ and $n \in \mathbb{N}$, we have a map 
\begin{equation}
    \label{eq: s1n round two}
    s^1(n):= s \circ \mathrm{can} : j_! f^s \mc{M}_U^{(n)} \rightarrow j_+ f^s \mc{M}_U^{(n)}.
\end{equation}
Assume $n$ is large enough so that $\coker s^1(n)$ has stabilised. There are two natural copies of the $\mc{D}_Y$-module $j_! \mc{M}_U$ appearing in \eqref{eq: s1n round two}: 
\begin{enumerate}
    \item as a quotient of the domain, 
    \[
     j_! f^s \mc{M}_U^{(n)} / s j_! f^s \mc{M}_U^{(n)}\simeq  j_! \mc{M}_U,
    \]
    \item as a submodule of the cokernel, 
    \[
   \ker (s: \Xi_f\mc{M}_U \rightarrow \Xi_f \mc{M}_U) \simeq  j_!\mc{M}_U. 
    \]
\end{enumerate}
Each of these copies of $j_! \mc{M}_U$ obtains a natural filtration from \eqref{eq: s1n round two} in the following way. 

\vspace{2mm}
\noindent
{\bf Filtration 1:} There is a filtration $S_\bullet$ on $j_+f^s \mc{M}_U^{(n)}$ given by powers of $s$:
\begin{equation}
\label{eq: powers of s filtration}
S_i:= s^{-i} j_+ f^s \mc{M}_U^{(n)}.
\end{equation}
where we set $s^{i}=\id$ for $i\leq 0$.
Pulling back the filtration $S_\bullet$ along the canonical map $\mathrm{can}$ \eqref{eq: canonical map}, we obtain a filtration $P_\bullet$ of $j_!f^s \mc{M}_U^{(n)}$:
\[
P_i:= \mathrm{can}^{-1}(s^{-i} j_+ f^s \mc{M}_U^{(n)}).
\]
The filtration $P_\bullet$ on $j_!f^s \mc{M}_U^{(n)}$ induces a filtration on the quotient $j_!f^s\mc{M}_U^{(n)}/sj_!f^s \mc{M}_U^{(n)}$. Call this filtration $A_\bullet$:
\begin{equation}
    \label{eq: filtration 1}
    A_\bullet := \text{ filtration on }j_!\mc{M}_U \text{ induced from }P_\bullet.
\end{equation}
\vspace{1mm}

\noindent
{\bf Filtration 2:} As explained in \S\ref{sec: the maximal extension functor}-\S\ref{sec: geometric Jantzen filtrations}, the maximal extension $\Xi_f \mc{M}_U = \coker s^1(n)$ has a monodromy filtration, which induces a geometric Jantzen filtration $J_{!,\bullet}$ on $\ker(s: \Xi_f \mc{M}_U \rightarrow \Xi_f \mc{M}_U)$. This filtration can be described in terms of powers of $s$ acting on $\Xi_f\mc{M}_U$ \eqref{eq: monodromy filtrations as powers of s}:
\begin{equation}
    \label{eq: filtration 2}
    J_{!,i}:= \ker s \cap \im s^{-i},
\end{equation}
where it is taken that $\im s^i = \Xi_f \mc{M}_U$ for $i \leq 0$

\vspace{2mm}
The two copies of $j_! \mc{M}_U$ in \eqref{eq: s1n round two} can be identified as follows. Because the canonical map $\mathrm{can}$ \eqref{eq: canonical map} is $s$-invariant, $sj_!f^s \mc{M}_U^{(n)}$ is the kernel of the composition 
\[
j_! f^s \mc{M}_U^{(n)} \xrightarrow{\mathrm{can}} j_+f^s\mc{M}_U^{(n)} \xrightarrow{q} \coker(s^1(n))=\Xi_f\mc{M}_U,
\]
where $q$ is the quotient of $j_+f^s\mc{M}_U^{(n)}$ by $\im s^1(n)$. Hence the canonical map descends to an injective map on the quotient, 
\begin{equation}
    \label{eq: can bar}
    \overline{\mathrm{can}}: j_!f^s \mc{M}_U^{(n)}/s j_! f^s \mc{M}_U^{(n)} \rightarrow \Xi_f \mc{M}_U.
\end{equation}
By construction, $\im \overline{\mathrm{can}} = \ker (s: \Xi_f\mc{M}_U \rightarrow \Xi_f \mc{M}_U)$. In this way, $\overline{\mathrm{can}}$ provides an isomorphism between the two copies of $j_! \mc{M}_U$ in \eqref{eq: s1n round two}. Moreover, because the geometric Jantzen filtration on $\ker(s: \Xi_f\mc{M}_U \rightarrow \Xi_f \mc{M}_U)$ can be realized in terms of the image filtration \eqref{eq: filtration 2}, it is clear that under $\overline{\mathrm{can}}$, Filtration \eqref{eq: filtration 1} and Filtration \eqref{eq: filtration 2} align. 

\begin{remark}
    The alignment of the two filtrations above is most easily seen through an example. In \cite{BohunRomanov}, a detailed $\mf{sl}_2(\C)$ example is given. We encourage the interested reader to view Figure 8 in that paper. 
\end{remark}

Now we specialise to the setting of \S\ref{sec: geometric Janten filtrations on Whittaker D-modules}. We will see that in these circumstances, Filtration 1 above aligns with the algebraic Jantzen filtration under the global sections functor, hence the global sections of the geometric Jantzen filtration (Filtration 2 above) also align with the algebraic Jantzen filtration. 

Set $Y=\widetilde{X}$ to be base affine space, $\widetilde{Q}$ to be the preimage under $\pi: \widetilde{X} \rightarrow X$ of an $N$-orbit admitting an $\eta$-twisted connection, and let $f_\varphi$ be the regular function on $\widetilde{X}$ given by \eqref{eq: condition of regular function} corresponding to a positive regular integral weight $\varphi \in \mf{h}^*$. As above, set $U:= f_\varphi^{-1}(\mathbb{A}^1 - \{0\})$. We have the following commuting inclusions of subvarieties:
\begin{equation}
    \label{eq: Qs in Us}
    \begin{tikzcd}
        U \arrow[rr, hookrightarrow, "j_U"] & & \widetilde{X} \\
        \widetilde{Q} \arrow[u, hookrightarrow, "k"] \arrow[urr, hookrightarrow, "i_{\widetilde{Q}}"] \arrow[rr, hookrightarrow, "j_{\widetilde{Q}}"] & &\overline{\widetilde{Q}} \arrow[u, hookrightarrow, "\overline{k}"]
    \end{tikzcd}
\end{equation}
The maps $k$ and $\overline{k}$ are closed immersions, and the maps $j_U$ and $j_{\widetilde{Q}}$ are open immersions. Applying the filtration construction above to the $\mc{D}_U$-module $k_!\mc{O}_{\widetilde{Q}, \eta} = k_+ \mc{O}_{\widetilde{Q}, \eta}$, we obtain from Filtration 1 a filtration of the $\mc{D}_{\widetilde{X}}$-module $j_{U!}f_\varphi^sk_! \mc{O}_{\widetilde{Q}, \eta}^{(n)}$ by pulling back the ``powers of s'' filtration along the canonical map. By the projection formula \cite[\S4]{D-modulesnotes}, 
\begin{equation}
    \label{eq: pushing commutes with deformation}
    j_{U!}f_\varphi^s k_! \mc{O}_{\widetilde{Q}, \eta}^{(n)} = i_{\widetilde{Q}!} f_\varphi^s \mc{O}_{\widetilde{Q}, \eta}^{(n)}.
\end{equation}
Hence under the equivalence \eqref{eq: two notions of monodromicity}, we obtain a filtration of the corresponding $\widetilde{\mc{D}}$-module $\pi_\bullet i_{\widetilde{Q}!} f_\varphi^s \mc{O}_{\widetilde{Q}, \eta}^{(n)}$. 

For $\lambda \in \h^*$, denote by 
\begin{equation}
    \label{eq: deformed standard D-module}
    \left( \pi_\bullet i_{\widetilde{Q}!} f_\varphi^s \mc{O}_{\widetilde{Q}, \eta}^{(n)}\right)_{\widehat{\lambda}}
\end{equation} 
the subsheaf of $\pi_\bullet i_{\widetilde{Q}!} f_\varphi^s \mc{O}_{\widetilde{Q}, \eta}^{(n)}$ annihilated by a power of $\mf{m}_\lambda \widetilde{\mc{D}} \subset \widetilde{\mc{D}}$, where $\mf{m}_\lambda \subset U(\mf{h})$ is the ideal defined in Remark \ref{rem: relationship to TDOs}. The subsheaf \eqref{eq: deformed standard D-module} is $\widetilde{\mc{D}}$-stable, and its global sections have the structure of a $\widetilde{U}$-module which is annihilated by the same power of $\mf{m}_\lambda \subset \widetilde{U}$ \cite[\S3.3]{BBJantzen}. In particular, for regular integral antidominant $\lambda \in \h^*$, the global sections of \eqref{eq: deformed standard D-module} can be considered as a $U(\mf{g})$-module with generalised infinitesimal character $\chi_\lambda$ \cite[\S3.1]{BBJantzen}, where $\chi_\lambda$ is as in \eqref{eq: inf char for g}. In fact, the global sections of the $\widetilde{\mc{D}}$-module \eqref{eq: deformed standard D-module} are a deformed Whittaker module.  
 \begin{lemma}
 \label{lem: global sections of deformed d-module are deformed whittaker}
Let $\lambda \in \mf{h}^*$ be a regular integral antidominant weight, $Q = C(w^C) \subset X$ a Bruhat cell corresponding to a longest coset representative in $C \in W_\eta \backslash W$, and $\widetilde{Q} = \pi^{-1}(Q) \subset \widetilde{X}$ the corresponding union of $N$-orbits in base affine space. For any $n \in \Z_{>0}$, 
\[
\Gamma\left(X, \left(\pi_\bullet i_{\widetilde{Q} !} f_\varphi^s\mc{O}_{\widetilde{Q}, \eta}^{(n)}\right)_{\widehat{\lambda}}\right) \cong M_{A^{(n)}}(w^C\lambda,\eta)
\]
as $(U(\mf{g}),A^{(n)})$-modules, where $A=\C[[s]]$, $A^{(n)}=A/s^n=\C[s]/s^n$, and $M_{A^{(n)}}(w^C\lambda,\eta)$ is the deformed Whittaker module (Definition \ref{def: deformed standard Whittaker}) with respect to the algebra homomorphism $\psi:U(\mf h)\to A^{(n)}$ mapping  $h\in \h$ to $\varphi(h)s.$ 
\end{lemma}
\begin{proof} The case $n=1$ is Theorem \ref{thm: matching of global sections undeformed}. An  adaptation of the proof in \cite[\S4]{Romanov} for $\widetilde{\mc{D}}$-modules can be used to prove the lemma for $n>1$.  Alternatively, the statement follows from the fact that both objects are uniquely determined as iterated extensions of the $n=1$ case. We expand on this perspective below. 

To ease notation, set 
\[
\mc{M}_\lambda^{(n)} := \left(\pi_\bullet i_{\widetilde{Q} !} f_\varphi^s\mc{O}_{\widetilde{Q}, \eta}^{(n)}\right)_{\widehat{\lambda}} \text{ and } M_\lambda^{(n)}:= M_{A^{(n)}}(w^C\lambda,\eta).
\]
Using adjunctions of $\mc{D}$-module functors and an $\eta$-twisted adaptation of the Langland's classification in \cite[\S2.6]{BBJantzen}, one can show that
\begin{equation}
    \label{eq: first ext agreement}
\operatorname{Ext}_{\widetilde{\mc{D}}}^1\left( \mc{M}_\lambda^{(1)}, \mc{M}_\lambda^{(n)} \right) = \operatorname{Ext}^1_{U(\h)}\left(A_{\lambda + \rho}^{(1)}, A_{\lambda + \rho}^{(n)}\right). 
\end{equation}
On the other hand, it follows from tensor-hom adjunction and the Harish-Chandra isomorphism for $\mc{Z}(\mf{l}_\eta)$ that 
\begin{equation}
\label{eq: second ext agreement}
\operatorname{Ext}^1_{U(\g)}\left(M_\lambda^{(1)}, M_\lambda^{(n)} \right) = \operatorname{Ext}^1_{U(\h)}\left(A_{\mu}^{(1)}, A_\mu^{(n)}\right)
\end{equation}
for $\mu = w^C \lambda - \rho + \rho_\eta$. Since both $\mu$ and $\lambda+\rho$ are regular, the second Ext groups in each equality \eqref{eq: first ext agreement} and \eqref{eq: second ext agreement} can be canonically identified. 

Pushforward along the natural $U(\h)$-module homomorphism $\iota:A^{(1)}_{\mu}\to A^{(n)}_{\mu}$ yields a map
\[
\iota_*:\operatorname{Ext}^1_{U(\mf h)}(A^{(1)}_{\mu}, A^{(1)}_{\mu})\to \operatorname{Ext}^1_{U(\mf h)}(A^{(1)}_{\mu}, A^{(n)}_{\mu}).
\]
Using the Koszul resolution of $A_\mu^{(1)}$, the domain of $\iota_*$ may be identified with $\h^*$. Now, both $\mc{M}_\lambda^{(n+1)}$ and $M_\lambda^{(n+1)}$ arise in \eqref{eq: first ext agreement} and \eqref{eq: second ext agreement} as the extensions corresponding to the element $\iota_*(\varphi)$. 

\end{proof}
A similar statement holds for deformed costandard objects, and these are compatible with filtrations.

From Lemma \ref{lem: global sections of deformed d-module are deformed whittaker}, we can see that the geometric Jantzen filtration of \S\ref{sec: geometric Jantzen filtrations} aligns with the algebraic Jantzen filtration of Definition \ref{def: Jantzen filtration of a standard Whittaker}. Specifically, we have shown that we can realise the geometric Jantzen filtration on the standard $(\mc{D}_{\widetilde{X}}, N, \eta)$-module 
$i_{\widetilde{Q}!} \mc{O}_{\widetilde{Q}, \eta}$ 
as the filtration obtained by pulling back the ``powers of $s$" filtration \eqref{eq: powers of s filtration} along the canonical map between deformed standard and costandard $(\mc{D}_{\widetilde{X}}, N, \eta)$-modules and then descending to the quotient 
\[
i_{\widetilde{Q}!} f_\varphi^s \mc{O}_{\widetilde{Q}, \eta}^{(n)} / s i_{\widetilde{Q}!} f_\varphi^s \mc{O}_{\widetilde{Q}, \eta}^{(n)} \simeq i_{\widetilde{Q}!} \mc{O}_{\widetilde{Q}, \eta}.
\]
By specifying a regular, antidominant and integral weight $\lambda \in \mf{h}^*$ and taking global sections, we obtain a filtration on the corresponding standard Whittaker module $M(w^C\lambda, \eta)$, which is given by pulling back the ``powers of $s$" filtration along the canonical map between standard and costandard deformed Whittaker modules. This is exactly how the algebraic Jantzen filtration is defined. By varying the $N$-orbit $Q=C(w^C)$, we obtain the result for all standard Whittaker modules $M(\mu, \lambda)$ corresponding to regular integral weights $\mu \in \mf{h}^*$. 
\section{Geometric proof of the Jantzen conjecture}
\label{sec: Proof of the Jantzen Conjecture}
We now show how the Jantzen conjecture for Whittaker modules follows from the theory of mixed twistor $\mc D$-modules.
\subsection{Background on mixed twistor modules}
\label{sec: background on mixed twistor modules}
We collect important properties of algebraic mixed twistor modules that we will need, following \cite{mochizukiMixedTwistorDmodules2015}.
\begin{itemize}
\item For a smooth complex variety $Y$, we denote by $\MTM(Y)$ the abelian category of (algebraic) mixed twistor $\mc D$-modules on $Y.$ 
\item For $w\in \Z,$ there is an abelian semisimple subcategory $\MT(Y,w)\subset \MTM(Y)$ of polarizable pure twistor $\mc D$-modules of weight $w.$ 
\item
A mixed twistor $\mc D$-module $M\in \MTM(Y)$ is equipped with an increasing filtration $W_\bullet$, called the \emph{weight filtration}, such that the associated graded $\operatorname{Gr}_w^W(M)\in \MT(Y,w)$ is pure of weight $w.$
\item Any morphism of mixed twistor modules is strict with respect to the weight filtration.
\item For $k\in\frac{1}{2}\mathbb{Z}$ the Tate twist $(k)$ on $\MTM(Y)$ is an automorphism the which shifts the weight filtration $W_wM(k)=W_{w-2k}M$ and maps $\MT(Y,w)$ to $\MT(Y,w-2k)$.
\item There is an exact realisation functor to $\mc D$-modules
$$\realmtm=\realmtm^{\lambda_0}: \MTM(Y)\to \mc{M}_{hol}(\mc{D}_Y).$$
The definition involves the choice of a generic $\lambda_0\in \C$ which we fix here.
\item All holonomic $\mc D$-modules that lift to $\MTM$ are hence equipped with a weight filtration. Moreover, all morphisms that lift to $\MTM$ are strict for these weight filtrations.
\item Denote by $\mc{M}_{hol}(\mc{D}_Y)_{\operatorname{ss}}\subset \mc{M}_{hol}(\mc{D}_Y)$ the subcategory of semisimple holonomic $\mc D$-modules. Then, for $w\in \Z$, the realisation functor restricts to a essentially surjective functor
$$\realmtm:\MT(X,w)\to \mc{M}_{hol}(\mc{D}_Y)_{\operatorname{ss}}.$$ 
\item The six operations for $\mc D$-modules
lift to $\MTM$ and commute with the realisation functor $\realmtm.$
\item Let $f:Y \rightarrow \A^1$ be a function on a smooth variety $Y$ and $j:U=f^{-1}(0)\to Y.$ Beilinson's maximal extension functor $\Xi_f$ for $\mc D$-modules, see Section \ref{sec: the maximal extension functor}, admits a lift
$$\Xi_f^{(a)}:\MTM(U)\to \MTM(X).$$
We obtain a lift for each $a\in\Z$ and the different lifts differ by a Tate twist.
\item The monodromy filtrations on $\Xi_f^{(a)}(M)$, $j_!(M)$ and $j_+(M)$ for $M\in \MTM(U)$ coincide with the respective weight filtrations. The associated graded of the monodromy filtration is hence a sum of pure twistor modules and, after applying $\realmtm,$ a semisimple holonomic $\mc D$-module.
\end{itemize}
\subsection{Geometric proof of the Jantzen conjectures} 
\label{sec: proof of the Jantzen conjecture} We explain how the constructions for $\mc D$-modules on flag varieties, see Section \ref{sec: D-modules}, lift to mixed twistor $\mc D$-modules and show how this implies the Jantzen conjecture Theorem \ref{thm: Jantzen conjecture Whittaker}.

For each longest coset representative $w^C$ for $C\in W_\eta\backslash W$ we obtain the $\mc D$-module
$$\widetilde{\mc{M}}_{C,\eta}=i_{\widetilde C(w^C)!}\pi^+\overline{\eta}^!\exp \in \mc{M}_{hol}(\mc{D}_{\widetilde C(w^C)}) \text{ for } \exp\in\mc{M}_{hol}(\mc{D}_{\A^1})_{\operatorname{ss}}.$$
 Here $\overline{\eta}: C(w^C) \rightarrow \A^1$ is defined as in Remark \ref{rem: the exponential D-module}. By lifting the exponential $\mc D$-module we obtain a lift as a mixed twistor module
$$\underline{\widetilde{\mc{M}}}_{C,\eta}=i_{\widetilde C(w^C)!}\pi^+\overline{\eta}^!\underline{\exp}(-\tfrac{1}{2}\ell(w))\in\MTM(\widetilde C(w^C)) \text{ for } \underline{\exp}\in\MT(\A^1,0).$$
The shift is chosen so that $\overline{\eta}^!\underline{\exp}(-\tfrac{1}{2}\ell(w^C))$ is pure of weight zero.

The weight filtration $W_\bullet$ on $\underline{\widetilde{\mc{M}}}_{C,\eta}$ (and thereby on $\widetilde{\mc{M}}_{C,\eta}=\realmtm(\underline{\widetilde{\mc{M}}}_{C,\eta})$) agrees with the monodromy filtration $J_{!,\bullet}$. Hence, the associated graded pieces of the filtration are semisimple.
Now let $\lambda\in \mf h^*$ be anti-dominant and regular. By Lemma \ref{lem: realising standard Whittakers on base affine space} passing to global sections and $\lambda$-invariants we obtain the standard Whittaker module
$$M(w^C\lambda, \eta)=\Gamma(\widetilde{X}, \widetilde{\mc{M}}_{C,\eta})_\lambda$$
and by Section \ref{sec: relationship between algebraic and geometric Janten filtrations}, the Jantzen filtration $M(w^C\lambda, \eta)_\bullet$ is induced by the monodromy filtration $J_{!,\bullet}$. Hence, the associated graded $\gr_i M(w^C\lambda, \eta)$ of the Jantzen filtration is semisimple which proves the first part of Theorem \ref{thm: Jantzen conjecture Whittaker}.

The second part of Theorem \ref{thm: Jantzen conjecture Whittaker} is about strictness of embeddings of standard Whittaker-modules with respect to Jantzen filtration. To prove this, we show that these embeddings lift to maps of mixed twistor modules.
\begin{lemma}
Let $w^{C_1},w^{C_2}$  be longest coset representatives for $C_1,C_2\in W_{\eta}\backslash W$ with $w^{C_1}<w^{C_2}$. Then the embedding $M(w^{C_1}\lambda,\eta)\hookrightarrow M(w^{C_2}\lambda,\eta)$ lifts to an element in
$$\Hom_{\MTM(\widetilde{X})}(\underline{\widetilde{\mc{M}}}_{C_1,\eta}, \underline{\widetilde{\mc{M}}}_{C_2,\eta}(\tfrac{1}{2}(\ell(w^{C_2})-\ell(w^{C_1})))).$$
\end{lemma}
\begin{proof} The proof is similar to \cite[Corollary 5.3.4]{BB81}. 

By induction, we may can assume that $w^{C_1}=w$ and $w^{C_2}=ws$ for a simple reflection $s$. We now explain how to further reduce to the case $w=e$ and $\eta=0$ which is shown in \cite[Corollary 5.3.4]{BB81}.
Let $N_w=N\cap wsN^-(ws)^{-1}$ and $P_s\subset G$ the parabolic subgroup associated to $s.$ Then multiplication yields an isomorphism
$$N_{ws}\times P_s/N\stackrel{\sim}{\to}\widetilde{C}(ws)\cup \widetilde{C}(w), (n,pN/N)\mapsto nwpN/N$$
With respect to this isomorphism
we obtain for $C_1',C_2'\in W_\eta\backslash W$ the cosets of $e$ and $s$, respectively,
\begin{align*}
    \underline{\widetilde{\mc{M}}}_{C_1,\eta}&=\underline{\mc O}_{ N_{ws},\eta}\boxtimes \underline{\widetilde{\mc{M}}}_{C_1',0}\text{ and }
    \underline{\widetilde{\mc{M}}}_{C_2,\eta}=\underline{\mc O}_{ N_{ws},\eta}\boxtimes \underline{\widetilde{\mc{M}}}_{C_2',0}.
\end{align*}
Here, the first factor $\underline{\mc O}_{ N_{ws},\eta}$ is the pullback of the exponential twistor module under $\eta: N_{ws}\to \mathbb{A}^1$. The second factors are the pushforwards of the constant twistor modules. We obtain our desired morphism as $\id\boxtimes \iota$, where $0\neq \iota$ is the embedding of $\underline{\widetilde{\mc{M}}}_{C_1',0}$ into $\underline{\widetilde{\mc{M}}}_{C_2',0}(\tfrac{1}{2}).$
\end{proof}
Passing to $\lambda$-invariant global sections and using the strictness of the weight filtration for maps of mixed twistor modules, we hence obtain
$$M(w^{C_1}\lambda,\eta)_{i}= M(w^{C_1}\lambda,\eta)\cap M(w^{C_2}\lambda,\eta)_{i-\ell(w^{C_2})+\ell(w^{C_1})}$$
which is the second part of Theorem \ref{thm: Jantzen conjecture Whittaker}.
\bibliographystyle{alpha}
\bibliography{Jantzen}
\end{document}